\documentclass[10pt]{article}
 \usepackage[a4paper, total={5.7in, 9.2in}]{geometry} 
\usepackage{graphicx} % Required for inserting images
\usepackage{amsthm}
\usepackage{amsfonts}
\usepackage{amssymb}
\usepackage{amsmath}
\usepackage{mathtools}
\usepackage[maxbibnames=99]{biblatex}
\usepackage{xcolor}
\usepackage{esint} %needed for averaged integrals
\usepackage[shortlabels]{enumitem}
\usepackage{titlesec}
\usepackage{blindtext}

\usepackage[citecolor=blue,colorlinks]{hyperref}

\makeatletter
\newcommand{\myitem}[1]{%
\item[#1]\protected@edef\@currentlabel{#1}%
}
\makeatother

\newtheorem{innercustomgeneric}{\customgenericname}
\providecommand{\customgenericname}{}
\newcommand{\newcustomtheorem}[2]{%
  \newenvironment{#1}[1]
  {%
   \renewcommand\customgenericname{#2}%
   \renewcommand\theinnercustomgeneric{##1}%
   \innercustomgeneric
  }
  {\endinnercustomgeneric}
}

\newcustomtheorem{customthm}{Theorem}
\newcustomtheorem{customlemma}{Lemma}

\newtheorem{thm}{Theorem}[section]
\newtheorem{corollary}[thm]{Corollary}
\newtheorem{lemma}[thm]{Lemma}
\newtheorem{proposition}[thm]{Proposition}

\newtheorem*{thm*}{Theorem}
\newtheorem*{corollary*}{Corollary}
\newtheorem*{lemma*}{Lemma}
\newtheorem*{proposition*}{Proposition}

\theoremstyle{definition}
\newtheorem{definition}[thm]{Definition}
\newtheorem*{definition*}{Definition}
\newtheorem{remark}[thm]{Remark}

\newtheorem*{remark*}{Remark}

\newcommand{\bb}[1]{\mathbb{#1}}
\newcommand{\ssf}[1]{\mathsf{#1}}

\newcommand{\lip}{\mathsf{lip}}
\newcommand{\Lip}{\mathsf{Lip}}

\newcommand{\Ric}{\mathsf{Ric}}

\newcommand{\Hess}{\mathsf{Hess}}
\newcommand{\RCD}{\mathsf{RCD}}
\newcommand{\length}{\mathrm{L}}

\newcommand{\aH}{\mathsf{H}}

\newcommand{\sd}{\mathsf{d}}

\newcommand{\R}{\mathsf{R}}
\newcommand{\de}{{\rm d}}
\newcommand{\m}{\mathfrak{m}}

\newcommand{\mres}{\mathbin{\vrule height 1.6ex depth 0pt width
0.13ex\vrule height 0.13ex depth 0pt width 1.3ex}}

\makeatletter
\newcommand{\customlabel}[2]{%
   \protected@write \@auxout {}{\string \newlabel {#1}{{#2}{\thepage}{#2}{#1}{}} }%
   \hypertarget{#1}{#2}
}
\makeatother

\title{On manifolds with almost non-negative Ricci curvature and integrally-positive $k^{th}$-scalar curvature}
\author{Alessandro Cucinotta \and Andrea Mondino}

\begin{document}

\maketitle

\begin{abstract}
    We consider manifolds with almost non-negative Ricci curvature and strictly positive integral lower bounds on the sum of the lowest $k$ eigenvalues of the Ricci tensor.

    If $(M^n,g)$ is a Riemannian manifold  satisfying such curvature bounds for $k=2$, then we show that $M$  is contained in a neighbourhood of controlled width of an isometrically embedded $1$-dimensional sub-manifold. From this, we deduce several metric and topological consequences: $M$ has at most linear volume growth and at most two ends, it has bounded 1-Urysohn width, the first Betti number of $M$ is bounded above by $1$, and there is precise information on elements of infinite order in  $\pi_1(M)$.

    If $(M^n,g)$ is a Riemannian manifold  satisfying such bounds for $k\geq 2$, then 
    we show that $M$ has at most $(k-1)$-dimensional behavior at large scales. 
    
    If $k=n={\rm dim}(M)$, so that the integral lower bound is on the scalar curvature, assuming in addition that the $(n-2)$-Ricci curvature is non-negative, we prove that the dimension drop at large scales improves to $n-2$. 
    
    From the above results we deduce topological restrictions, such as upper bounds on the first Betti number.
\end{abstract}

\section*{Contents}
\contentsline {section}{\numberline {1}Introduction}{1}{section.1}%
\contentsline {section}{\numberline {2}Preliminaries}{6}{section.2}%
\contentsline {section}{\numberline {3}Integral estimates on scalar curvature and first geometric applications}{9}{section.3}%
\contentsline {subsection}{\numberline {3.1}Proof of Theorem \ref {CT1}}{9}{subsection.3.1}%
\contentsline {subsection}{\numberline {3.2} \parbox[t]{0.8\textwidth}{First geometric applications in almost non-negative Ricci and integrally-positive scalar curvature}}{15}{subsection.3.2}%
\contentsline {section}{\numberline {4}Thin metric spaces}{23}{section.4}%
\contentsline {section}{\numberline {5}\parbox[t]{0.8\textwidth}{Large scale geometry of almost non-negative Ricci and integrally-positive  scalar curvature}}{31}{section.5}%
\contentsline {subsection}{\numberline {5.1}Sufficient curvature conditions for a manifold to be thin}{31}{subsection.5.1}%
\contentsline {subsection}{\numberline {5.2}Large scale geometry of thin manifolds}{34}{subsection.5.2}%

\section{Introduction}
A renowned conjecture  by Yau asserts that 
if $(M^n,g,p)$ is a pointed Riemannian manifold of dimension $n$ with $\Ric_M \geq - \lambda$, then 
\[
\int_{B_1(p)} \R \, \de \ssf{Vol} \leq C(\lambda,n),
\]
where $\R$ denotes the scalar curvature. Under the assumption of non-negative sectional curvature, the aforementioned bound on the scalar curvature was obtained by Petrunin in \cite{PetruninYau}. A non-collapsed version of Yau's problem was then proposed by Naber in \cite{Naberconjectures}, where he conjectured that if $(M^n,g,p)$ is a pointed Riemannian manifold with $\Ric_M \geq - \lambda$ and $\ssf{Vol}(B_1(p)) \geq v$, then $\fint_{B_1(p)} \R \, \de \ssf{Vol} \leq C(\lambda,n,v)$.
If true, Naber's version of Yau conjecture would allow to define a scalar curvature measure on non-collapsed Ricci limit spaces (see again \cite{Naberconjectures}).

These conjectures are still open and, more generally, the geometric interplay between Ricci and scalar curvature bounds is still not fully understood. 
A step forward in this direction was carried out in a series of recent papers
showing that $3$-manifolds with non-negative Ricci curvature and scalar curvature bounded from below by a strictly positive constant have at most linear volume growth at infinity (see \cite{wangvolumethreemanifolds, Chodoshvolumethreemanifolds}).
This is a particular case of a more general conjecture by Gromov  \cite{LargeGromov} stating that a non-compact pointed Riemannian manifold $(M^n,g,p)$ of dimension $n$ with non-negative Ricci curvature and scalar curvature bounded below by a strictly positive constant should satisfy
\begin{equation} \label{GVC}
\limsup_{t \to + \infty} \frac{ \ssf{Vol}(B_t(p))}{t^{n-2}} < + \infty.
\end{equation}

Without the condition on the Ricci curvature, manifolds with uniformly positive scalar curvature can have arbitrarily large volume growth (by the surgery results in \cite{GromovLawsonSurgery1,GromovLawsonSurgery2}).
Nevertheless such $3$-manifolds have been shown to have bounded Urysohn $1$-width (see \cite{Katzscalar,GromovScalar,Liokumovitchscalar,liokumovitchwang} and, assuming also non-negative Ricci curvature, \cite{ZhuGeometry}). This class of results enters in the framework of another more general conjecture by Gromov  \cite{GromovScalar}, stating that a manifold of dimension $n$ with scalar curvature bounded below by a strictly positive constant should have bounded Urysohn $(n-2)$-width.
\\

Motivated and inspired by the above conjectures and results, we next discuss the framework and the main findings of the present work.
\\
Given a manifold $(M^n,g)$ and $k \in \bb{N}$, $k \leq n$, we denote by $\R_k$ the sum of the lowest $k$ eigenvalues, counted with their multiplicity, of the Ricci tensor $\Ric_M$. In analogy with the terminology used in \cite{ShenKricci}, if $\R_k \geq c$, we say that the $k^{th}$-scalar curvature of $M$ is bounded below by $c$ (note that this is different from the convention used in \cite{WolfsonKricci}, where the Ricci curvature is said to be $k$-positive if $\R_k \geq 0$).
In this paper, we study the geometry of manifolds with almost non-negative Ricci curvature and $k^{th}$-scalar curvature bounded from below in integral sense by a positive constant (see \cite{Moull1,Moull2,Moull3,Moull4,Moull5} for previous literature on manifolds with $\R_k \geq c$).
As a byproduct, we also obtain new results on manifolds with almost non-negative Ricci curvature and a positive lower integral bound on the scalar curvature.

Broadly speaking, we show that a lower Ricci curvature bound combined with a positive lower bound on the $k^{th}$-scalar curvature forces the manifold to have $(k-1)$-dimensional behavior at large scales. Moreover, if $k$ coincides with the dimension of the manifold (so that the lower bound is on the scalar curvature), the dimensional bound on the geometry at large scales can be improved to $(k-2)$, matching in this way the dimension-drop appearing in the aforementioned conjectures by Gromov.
Our approach 
relies on techniques from metric geometry in which a limit space arises from a sequence of manifolds (see the foundational papers \cite{ChCo1, Colding1,ChcoStructure2,ChCo3}).
\\

The first main result of the paper is Theorem \ref{CT1}, which shows the continuity of certain integral functionals of $\R_k$ under Gromov-Hausdorff convergence. 
We recall that, if Yau conjecture holds, given a sequence $(M_i,g_i,p_i) \to (X,\sd,x)$,
one can define a scalar curvature measure $\mu$ on $X$ as a weak limit $\R_{M_i} \ssf{Vol}_i \to \mu$ (see \cite[Conjecture 2.18]{Naberconjectures}). Hence, continuity results for $\int \R$ under GH convergence would ensure that, if $X$ is itself a smooth manifold, then $\mu=\R_{X} \ssf{Vol}$. Theorem \ref{CT1} proves a result along these lines when $X=\bb{R}^n$ (see Remark \ref{R5}).
Other continuity results for the integral of the scalar curvature were previously obtained in the presence of a lower bound on the sectional curvature in \cite[Theorem B3]{PetruninPoli} (after Perelman) and \cite{lebedeva2022curvaturetensorsmoothablealexandrov}. 
We refer to Definition \ref{D1} for the definition of $(k,\delta)$-symmetric and $\delta$-regular balls in a manifold. Given $a,b \in \bb{R}$, we denote $a \wedge b:=\min\{a,b\}$ and $a \vee b:=\max\{a,b\}$.

\begin{thm} \label{CT1}
    Let $L,k,n \in \bb{N}$ with $k \leq n$ and $s \in (0,1)$ be fixed. Then 
    \begin{align*}
    &\sup \Big\{\fint_{B_1(x)} \R_{k} \wedge L \, \de \ssf{Vol} : (M,g,x) \text{ has } \Ric_{M} \geq -\delta, \text{ and } B_{10}(x) \text{ is  $(k,\delta)$-symmetric} 
    \Big\} ,
    \\
    &  \sup \Big\{\fint_{B_1(x)} |\R|^s \, \de \ssf{Vol} : (M,g,x) \text{ has } \Ric_{M} \geq -\delta, \text{ and } B_{10}(x) \text{ is } \delta \text{-regular} 
    \Big\}  
    \end{align*}
    converge to $0$ as $\delta \searrow 0$.
    \end{thm}

The previous result is the technical tool that will be used to study manifolds with a positive uniform lower bound on $\R_k$. 
\\
\\
As a first application of Theorem \ref{CT1}, in Theorem \ref{thmIntMain} we show that manifolds of any dimension with integrally positive $\R_2$ (and almost non-negative Ricci curvature) are contained in a neighbourhood of controlled width of an isometrically embedded $1$-dimensional manifold. 

The result applies also to $3$-manifolds with integrally positive scalar curvature and, assuming either an integral sectional curvature lower bound or a strong non-collapsing assumption, gives a new precise description of the macroscopic behavior of these spaces. Let $v,\epsilon, \delta,L \in (0,+\infty)$, $n \in \bb{N}$, $s \in (0,1)$, and consider the following sets of conditions on a manifold $(M^n,g)$.

\begin{enumerate}[label = (\roman*)]
        \item \label{Condition1Int}  $n\geq 2$, $\ssf{Ric}_{M} \geq -\delta$, and
    \begin{equation*} %\label{eq:Condition1Int}
     \fint_{B_1(x)} \R_{2} \wedge L \,\de\ssf{Vol}  \geq \epsilon, \quad \forall x \in M.
    \end{equation*}
    \item\label{condition2Int} $n=3$, $\ssf{Ric}_{M} \geq -\delta$, and
    \begin{equation*} %\label{eq:condition2Int}
    \fint_{B_1(x)} \ssf{Sec}_{M} \wedge 0 \, \de \ssf{Vol} \geq -\delta, \quad  \fint_{B_1(x)} \R_{M} \wedge L \,\de\ssf{Vol}  \geq \epsilon, \quad \forall x \in M.
    \end{equation*}
    \item \label{condition3Int} $n=3$, $\ssf{Ric}_{M} \geq -\delta$, and
    \begin{equation*} %\label{eq:condition3Int}
     \fint_{B_1(x)} |\R_{M}|^s \,\de\ssf{Vol}  \geq \epsilon, \quad \ssf{Vol}(B_1(x)) \geq v, \quad \forall x \in M.
    \end{equation*}
    \end{enumerate}

Given $r>0$ and a subset $A \subset M$, the $r$-neighbourhood of $A$ in $M$ is the set of points whose distance from $A$ is less than $r$.

\begin{thm} \label{thmIntMain}
Let $v,\epsilon,L \in (0,+\infty)$, $n \in \bb{N}$, $s \in (0,1)$ be fixed.
    There exist $C,\delta>0$ with the following property. If a manifold $(M^n,g)$ satisfies one of the conditions \ref{Condition1Int}, \ref{condition2Int}, \ref{condition3Int}, then
    there exists a $1$-dimensional connected manifold $I$ (possibly with boundary) and a distance preserving map $\phi:I \to M$, such that $M$ is contained in the $C$-neighbourhood of $\phi(I)$.
    \end{thm}

Theorem \ref{thmIntMain} has far reaching implications both for the metric and the topological properties of manifolds satisfying \ref{Condition1Int}, \ref{condition2Int}, or \ref{condition3Int}. These are collected in Theorem \ref{thmInt2} below.
We denote by $\mathrm{b}_1(M)$ the first Betti number of $M$.
Given a loop $\gamma$ in $M$, we denote by $[\gamma]$ its equivalence class in the fundamental group $\pi_1(M)$. We recall that a metric space $(X,\sd)$ has $1$-Urysohn width $\leq d$ if there exists a $1$-simplex $Y$, and a continuous map $\pi:X \to Y$, such that $\mathrm{diam}(f^{-1}(y)) \leq d$ for every $y \in Y$.

\begin{thm} \label{thmInt2}
    Let $v,\epsilon,L \in (0,+\infty)$, $n \in \bb{N}$, $s \in (0,1)$ be fixed.
    There exist $C,\delta>0$ with the following property. Let $(M^n,g)$ be a manifold satisfying one of the conditions \ref{Condition1Int}, \ref{condition2Int}, \ref{condition3Int}. Then, 
    \begin{enumerate}
    \item \label{point1}
        $\sup_{x \in M} \ssf{Vol}(B_t(x)) \leq Ct$, for all $t>0$.
       \item \label{point2}
       If $\Ric_M \geq 0$, then $\inf_{x \in M} \ssf{Vol}(B_1(x))>0$.
       \item $M$ has $1$-Urysohn width $\leq C$.
       \item $M$ has at most two ends.
    	\item $\mathrm{b}_1(M) \leq 1$.
    	\item $\pi_1(M)$ is infinite if and only if $M$ is compact and its universal cover is non-compact.
    	\item If there exists a loop $\gamma \subset M$ such that $[\gamma] \in \pi_1(M)$ has infinite order, then $M$ is contained in a $C$-neighbourhood of $\gamma$.
       \end{enumerate}
       \end{thm}

        The previous two theorems are part of a series of recent results where positive lower bounds on the (scalar) curvature imply $1$-dimensional behavior at large scales both from the metric and the topological viewpoint. For instance, a linear volume growth bound for $3$-manifolds with $\Ric \geq 0 $ and $\R \geq 1$ was obtained in \cite{wangvolumethreemanifolds}, and was then generalized in \cite{ Chodoshvolumethreemanifolds, antonelli2024newspectralbishopgromovbonnetmyers, zhou2024optimalvolumeboundvolume}. Similarly, $3$-manifolds with $\R \geq 1$ have been recently shown to have bounded Urysohn $1$-width (see \cite{Katzscalar,GromovScalar,Liokumovitchscalar,liokumovitchwang}).
        
       We highlight that an analog of Theorem \ref{thmIntMain} is not to be expected in the higher dimensional case (see Remark \ref{R7}) and that the result fails without the lower bounds on the Ricci curvature (see Remark \ref{R6}). The proof of Theorem \ref{thmIntMain} relies on the study of a geometric condition on length minimizing geodesics in a metric space that was first studied in \cite{TrianglesKapo} (see Section \ref{S1}).
\\

We now turn our attention to manifolds with an integral lower bound on $\R_k$, possibly for $k \geq 3$.
Theorem \ref{thmInt3} shows that non-negative Ricci curvature, coupled with a sufficiently slow asymptotic integral decay of $\R_k$, implies a dimension-drop for the tangent cones at infinity. An analogous result (under different assumptions) was recently proved in \cite[Theorem 1.6]{zhu2024twodimensionvanishingsplittingpositive}.
We recall that if a pointed non-compact manifold $(M^n,g,p)$ has non-negative Ricci curvature, then Gromov's pre-compactness theorem ensures that, for every sequence $(M^n,g/r_j,p)$ as $r_j \to + \infty$,  there exists a subsequence converging in pointed Gromov-Hausdorff sense to a metric space $(X,\sd,p_\infty)$. Such  space obtained via blow-down is called a \emph{tangent cone at infinity} of $M$. 
Finally, we refer to Section \ref{S4} for the definition of the $(n-2)^{th}$-Ricci curvature $\Ric_{n-2}$ (after \cite{ShenKricci}).

\begin{thm} \label{thmInt3}
    Let $(M^n,g,p)$ be a Riemannian manifold and let $L>0, \;\alpha \in (0,2)$.
    \begin{enumerate}
        \item
        Assume that $\Ric \geq 0$ and
    \[
    \lim_{r \to + \infty} \fint_{B_r(p)} 0 \vee (r^{2-\alpha}\R_k) \wedge L \,\de\ssf{Vol}=L.
    \]
    Let $(X,\sd,x_\infty)$ be any tangent cone at infinity of $M$. 
    If $(\bb{R}^d,\sd_{eu})$ is a tangent space to $X$ at a point $x \in X$, then $d \leq k-1$. 
    \item 
    Assume that $\Ric_{n-2} \geq 0$ and
    \[
    \lim_{r \to + \infty} \fint_{B_r(p)} 0 \vee (r^{2-\alpha}\R) \wedge L \,\de\ssf{Vol}=L.
    \]
    Let $(X,\sd,x_\infty)$ be any tangent cone at infinity of $M$. 
    If $(\bb{R}^d,\sd_{eu})$ is a tangent space to $X$ at a point $x \in X$, then $d \leq n-2$.
    \end{enumerate}
    \end{thm}

    The previous theorem is sharp in the sense that it fails for $\alpha=0$ (see Remark \ref{R3}). A biproduct of Theorem \ref{thmInt3} is that, if a manifold has non-negative Ricci curvature and sufficiently slow integral asymptotic decay of the Ricci tensor, then it is compact (see Remark \ref{R4}).

We now consider the topological implications of non-negative Ricci curvature and asymptotically positive $k^{th}$-scalar curvature.
This is the content of Theorems \ref{thmIntFinal} and \ref{thmInt4} (compare with \cite[Theorems 1.3 and 1.7]{zhu2024twodimensionvanishingsplittingpositive}, after \cite{PanBlowDown}). 

\begin{thm} \label{thmIntFinal}
    Let $(M^n,g,p)$ be a non-compact Riemannian manifold with $n \geq 3$ and let $\alpha \in (0,2)$. Let $2 \leq k \leq n$ be a natural number.
    \begin{enumerate}
        \item 
        If $\Ric_M \geq 0$ and
    \begin{equation*} 
    \limsup_{r \to + \infty} \Big( r^{2-\alpha} \min_{B_r(p)} \R_k \Big) >0,
    \end{equation*}
    then $\mathrm{b}_1(M) \leq k-2$. 
    \item 
    If $\Ric_{n-2} \geq 0$ and
    \begin{equation*} 
    \limsup_{r \to + \infty} \Big( r^{2-\alpha} \min_{B_r(p)} \R \Big) >0,
    \end{equation*}
    then $\mathrm{b}_1(M) \leq n-3$.
    \end{enumerate}
    \end{thm}
    The proof of Theorem \ref{thmIntFinal} follows the lines of  \cite[Theorem 1.3]{zhu2024twodimensionvanishingsplittingpositive} by replacing the use of \cite[Theorem 1.6]{zhu2024twodimensionvanishingsplittingpositive} with Theorem \ref{thmInt3}.
    We remark that Theorem \ref{thmIntFinal}, unlike \cite[Theorem 1.3]{zhu2024twodimensionvanishingsplittingpositive}, allows for the scalar curvature to converge to zero at infinity.

\begin{thm} \label{thmInt4}
    Let $\epsilon,s,v \in (0,1)$, let $L,D \in (0,+\infty)$, and let $k,n \in \bb{N}$ with $k \leq n$. There exists $\delta>0$ such that for every manifold $(M^n,g,p)$ with
    $\Ric_{M} \geq -\delta$, $\mathrm{diam}(M) \leq D$, $\ssf{Vol}(B_1(p)) \geq v$, the following holds.
    \begin{enumerate}
    \item If
    $
     \fint_{B_1(p)} \R_{k} \wedge L \, \de \ssf{Vol} \geq \epsilon,
    $
    then $\mathrm{b}_1(M) \leq k-1$. 
        \item If
    $
     \fint_{B_1(p)} |\R|^s \, \de \ssf{Vol} \geq \epsilon,
    $
    then $\mathrm{b}_1(M) \leq n-2$. 
    \end{enumerate}
    \end{thm}

    If the manifold in the previous theorem has non-negative Ricci curvature, the previous result follows immediately from the splitting theorem \cite{SplittingCheegerGromoll}. We remark that Theorem \ref{thmInt4}, unlike \cite[Theorem 1.7]{zhu2024twodimensionvanishingsplittingpositive}, only requires a local bound on the scalar curvature.
\\
\\
Finally, we mention that Theorem \ref{CT1} can be combined with the results of \cite{CheegerNaber} to prove a result from Jiang and Naber (see Theorem \ref{TJN} below). This result was announced in \cite{Naberconjectures} without proof; we  sketch one in the appendix.

\begin{thm}[Jiang-Naber] \label{TJN}
    Let $K \in \bb{R}$, $n \in \bb{N}$, $v>0$, $s \in (0,1)$ be fixed and let $(M^n,g,p)$ be a manifold with $\Ric_M \geq K$ and $\ssf{Vol}(B_1(p)) \geq v$. There exists a constant $C(K,n,v,s)>0$ such that
    \[
    \fint_{B_1(p)} |\Ric|^s \,\de\ssf{Vol} \leq  C(K,n,v,s).
    \]
\end{thm}

The previous theorem is the best available result concerning the non-collapsed Yau conjecture mentioned at the begininning of the introduction. 
\\
\\
The paper is organized as follows.
Section \ref{S2} contains preliminaries. In Section \ref{S3}, we prove Theorems \ref{CT1}, \ref{thmInt3}, \ref{thmIntFinal}, and \ref{thmInt4}. In Section \ref{S1}, we study a property of length-minimizing geodesics in metric spaces which is then used to prove Theorem \ref{thmIntMain}. Finally, in Section \ref{Sec:MainGeom}, we prove Theorems \ref{thmIntMain} and \ref{thmInt2}. 

\subsection*{ Acknowledgments}  
The authors wish to thank the anonymous referee for the careful reading and constructive comments that allowed us to improve the manuscript.

A.\,M.\;acknowledges support from the European Research Council (ERC) under the European Union's Horizon 2020 research and innovation programme, grant agreement No.\;802689 ``CURVATURE''.

A.\;C.\;wishes to thank Xingyu Zhu for useful discussions on the content of this paper and, in particular, for pointing out the reference \cite{zhu2024twodimensionvanishingsplittingpositive}, whose content led to Theorem \ref{thmIntFinal}.

Part of this research was carried out at the Hausdorff Institute of Mathematics in Bonn, during the trimester program  ``Metric Analysis''. The authors wish to express their appreciation to the  institution for the stimulating atmosphere, and they acknowledge support  by the Deutsche Forschungsgemeinschaft (DFG, German Research Foundation) under Germany's Excellence Strategy – EXC-2047/1 – 390685813. 

For the purpose of Open Access, the authors applied a CC BY public copyright licence to any Author Accepted Manuscript (AAM) version arising from this submission.

\section{Preliminaries} \label{S2}
We start by recalling some basic notation and terminology about metric spaces. Throughout the section,  $(X,\sd)$ denotes a proper, separable, length metric space.
A metric measure spaces $(X,\sd,\m)$ is a triple, where $(X,\sd)$ is a metric space and $\m$ is a non-negative Borel measure on $X$ which is finite on bounded sets and whose support is the whole $X$. We denote by $\aH^k$ the $k$-dimensional Hausdorff measure induced by $\sd$ on $X$.
Given a smooth manifold $(M,g)$ we denote by $\sd$ and $\ssf{Vol}$ respectively the Riemannian distance and volume measure induced by $g$.
Given an open set $\Omega \subset X$, we denote by $\Lip(\Omega)$, $\Lip_{\rm loc}(\Omega)$ and $\Lip_c(\Omega)$ respectively the set of  Lipschitz, locally Lipschitz and compactly supported Lipschitz functions in $\Omega$. We denote the \emph{slope} of $f \in \Lip_{\rm loc}(\Omega)$ at $x \in \Omega$ by
\[
\lip(f)(x):= \limsup_{y \to x} \frac{|f(x)-f(y)|}{\ssf{d}(x,y)}.
\]

We say that two pointed metric measure spaces $(X,\sd_X,\m_X,\bar{x})$ and $(Y,\sd_Y, \m_Y,\bar{y})$ are \emph{isomorphic} if there exists a bijective distance preserving map $i: X \to Y$ such that $i_{\#}\m_X = \m_Y$ and $i(\bar{x}) = \bar{y}$. 

We next recall the main definitions concerning the pointed Gromov-Hausdorff convergence, referring to \cite{YuBurago_1992}, \cite{Vil}, and \cite{GMS13} for an overview on the subject. 

\begin{definition}[$\delta$-GH maps]
    Let $(X,\sd_X,\bar{x})$ and $(Y,\sd_Y,\bar{y})$ be pointed metric spaces and let $\delta >0$. A map $f:X \to Y$ is a \emph{$\delta$-GH map} if:
    \begin{itemize}
   \item    $f(\bar{x})=\bar{y}$;
   \item the image of $f$ is a $\delta$-net in $Y$, i.e. for every $y\in Y$ there exists $x\in X$ such that $\sd_Y(y, f(x))\leq \delta$;  
     \item  the distortion of $f$ is less than $\delta$, i.e. 
     \[\sup_{x_0,x_1 \in X} \Big|\sd_X(x_0,x_1)-\sd_Y(f(x_0),f(x_1)) \Big| \leq \delta.
        \]
        \end{itemize}
\end{definition}

\begin{definition}[pGH-convergence] \label{D2}
    A sequence of pointed metric spaces $(X_j,\sd_j,\bar{x}_j)$ converges to $(X,\sd,\bar{x})$ in \emph{pointed Gromov-Hausdorff-sense} (pGH, for short) if for every $\delta,R>0$ there exists $N>0$ such that, for every $j \geq N$, there exists a $\delta$-GH map $f_j^{R}:(\bar{B}_R(\bar{x}_j),\sd_j,\bar{x}_j) \to (\bar{B}_R(\bar{x}),\sd,\bar{x})$.
\end{definition}

The previous definition of pointed Gromov Hausdorff convergence is slightly different from the one given in \cite[p. 272]{YuBurago_1992}. Nevertheless, they coincide when considering the convergence of length spaces (as it will be, throughout the paper).

\begin{definition}[pmGH-convergence]
    We say that a sequence of pointed metric measure spaces $(X_j,\sd_j, \m_j, \bar{x}_j)$ converges to $(X,\sd,\m,\bar{x})$ in \emph{pointed measured Gromov-Hausdorff-sense} (pmGH, for short) if it converges in pointed Gromov-Hausdorff-sense and, for every $R>0$, the maps $f_j^{R}:\bar{B}_R(\bar{x}_j) \to \bar{B}_R(\bar{x})$ given by Definition \ref{D2} satisfy $(f_j^{R})_{\#} (\m_j \mres \bar{B}_R(\bar{x}_j)) \to \m \mres \bar{B}_R(\bar{x}) $, as $j\to \infty$, weakly in duality with continuous boundedly supported functions on $X$.
\end{definition}

Both pointed Gromov-Hausdorff and pointed measured Gromov-Hausdorff convergence are metrizable. The corresponding distances are denoted by $\sd_{{\rm pGH}}$ and $\sd_{{\rm pmGH}}$.
We recall that, in the case of a sequence of uniformly locally doubling metric measure spaces $(X_j,\sd_j, \m_j,\bar{x}_j)$
(as in the case of $\RCD(K, N)$ spaces), pointed measured Gromov-Hausdorff convergence to $(X,\sd,\m,x)$ can be equivalently characterized by asking for the existence of a proper
metric space $(Z, \sd_z )$ such that all the metric spaces $(X_j
, \sd_j)$ are isometrically embedded
into $(Z, \sd_z )$, $\bar{x}_j \to x$ and $\m_j \to \m$ weakly in duality with continuous boundedly supported functions in $Z$ (see \cite{GMS13}). 

Next, we recall the fundamental notion of  $(k,\delta)$-splitting map (see \cite{ColdingConv, ChCo1, Colding1, CheegerNaberCodim4}) on a Riemannian manifold with Ricci curvature bounded below. 

\begin{definition}[$(k,\delta)$-splitting maps]\label{def:delta-split-map}
Let $(M^n,g)$ be a Riemannian manifold and let $1\leq k\leq n$. A harmonic map $u:B_r(x) \to \bb{R}^k$ is a \emph{$(k,\delta)$-splitting map} if:
\begin{enumerate}
    \item For every $i,j \in \{1, \cdots,k\}$, it holds
$
\fint_{B_r(x)} | \nabla u_i \cdot \nabla u_j-\delta_{ij}| \, \de \ssf{Vol} \leq \delta.
$
\item $\sup_{B_r(x)} |\nabla u| \leq 1 +\delta$.
\item $r^2 \fint_{B_r(x)} |\Hess(u)|^2 \,\de\ssf{Vol} \leq \delta^2$.
\end{enumerate}
\end{definition}

\begin{definition}[$(k,\delta)$-symmetric balls] \label{D1}
    Let $(M^n,g)$ be a Riemannian manifold. A ball $B_r(p) \subset M$ is said to be \emph{$(k,\delta)$-symmetric} if there exists a pointed metric space $(Y,\sd_Y,\bar{y})$ such that $\sd_{{\rm GH}}(B_r(p),B^{\bb{R}^k \times Y}_r(0,\bar{y})) \leq \delta r$. If a ball is $(n,\delta)$-symmetric, we will say that it is $\delta$-regular.
\end{definition}

The next theorem corresponds to \cite[Theorem $4.11$]{CJN} (after \cite{ChCo1} and \cite{Colding1}).

\begin{thm} \label{T:Fund1}Let $n \in \bb{N}$.
For every $\epsilon>0$, there exists $\delta_0(\epsilon,n)>0$ with the following property. Let $\delta\in (0,\delta_0)$ and let $(M^n,g,p)$ be a pointed Riemannian manifold with $\Ric_{M} \geq -\delta$.
    \begin{enumerate}
        \item If there exists a $(k,\delta)$-splitting map $u:B_2(p) \to \bb{R}^k$, then $B_1(p)$ is $(k,\epsilon)$-symmetric.
        \item \label{ItemThmFund2}
        If $B_2(p)$ is $(k,\delta)$-symmetric, then there exists a $(k,\epsilon)$-splitting map $u:B_1(p) \to \bb{R}^k$.
    \end{enumerate}
\end{thm}

The next result was proved in \cite{ChCo1} (see also \cite{AMSbakry}).

\begin{thm} \label{TGoodCut}
    Let $K \in \bb{R}$ and $n \in \bb{N}$. Let $(M^n,g)$ be a Riemannian manifold with $\Ric_M \geq K$. For every $p \in M$ and $R>2r>0$, there exists a function $\phi \in C^{\infty}_c(B_R(p))$ which is identically equal to $1$ on $B_r(p)$, with support contained in $B_{2r}(p)$, and such that
    \[
    r^2 |\Delta \phi|+r|\nabla \phi| \leq C(K,n,R).
    \]
\end{thm}

The next theorem follows from \cite{Colding1,ColdingConv,ChCo1}.

\begin{thm} \label{T1} Let $n \in \bb{N}$.
    For every $\epsilon>0$ there exists $\delta=\delta(n, \epsilon)>0$ such that if $(M^n,g)$ has $\Ric_M \geq -\delta$ and $B_{10}(x) \subset M$ is $\delta$-regular, then $B_r(y)$ is $\epsilon$-regular for every $r \in (0,5)$ and every $y \in B_5(x)$.
\end{thm}

\begin{definition}[Ricci limit space]
A pointed metric measure space $(X,\sd,\m,\bar{x})$ is a \emph{Ricci limit space} if there exists a sequence of pointed Riemannian manifolds $(M^n_j,g_j,p_j)$ with $\Ric_{M_j} \geq K$ for some fixed $K \in \bb{R}$ and $n \in \bb{N}$, such that $$(M^n_j,\sd_{g_j}, \ssf{Vol}_j(B_1(p_j))^{-1}\, \ssf{Vol}_j,p_j) \longrightarrow (X,\sd, \m, \bar{x}) \quad \text{in pmGH-sense.}
$$

The limit space is said to be \emph{non-collapsed} if $\ssf{Vol}_j(B_1(p_j)) \geq v>0$ for some fixed $v>0$.
\end{definition}

The next theorem was proved in  \cite{ColdingConv, Colding1}.

\begin{thm} \label{T3}
    Let $(M^n_j,g_j,p_j)$ be a sequence of pointed Riemannian manifolds with a uniform lower bound on the Ricci curvature converging in pGH-sense to a metric space $(X,\sd,\bar{x})$. Then precisely one of the following holds.
    \begin{enumerate}
        \item Either:
        $
        \limsup_{j \to + \infty} \ssf{Vol}_j(B_1(p_j))>0.
        $
        In this case the $\limsup$ is a limit and the pGH convergence can be improved to pmGH convergence to $(X,\sd,\aH^n,\bar{x})$.
        \item  Or:
        $
        \lim_{j \to + \infty} \ssf{Vol}_j(B_1(p_j))=0.
        $
        In this case the Hausdorff dimension of $X$ is bounded above by $n-1$.
    \end{enumerate}
\end{thm}

\begin{definition}[Tangent space]
    Let $(X,\sd,\m,\bar{x})$ be a Ricci limit space. A metric space $(Y,\sd_Y,y)$ is a \emph{tangent space} at $\bar{x}$ if there exists a sequence $1 > r_j > 0$, $r_j \searrow 0$ that satisfies 
    $$
    (X, \sd / r_j, \bar{x}) \stackrel{\rm{pGH}}{\longrightarrow} (Y, \sd_Y, y).
    $$
\end{definition}

Gromov pre-compactness theorem implies that the set of tangent spaces at a point is always non-empty. Moreover, in \cite{Colding1} it was shown that for non-collapsed Ricci limit spaces all tangent spaces are metric cones.

\begin{definition}[$k$-singular stratum]
    Let $(X,\sd, \m, \bar{x})$ be a Ricci limit space. For every $k \in \bb{N}$, the \emph{$k$-singular stratum} is defined to be 
    \begin{equation*}
        {\mathcal S}_{k} = \left \{x \in X : \ \mbox{no tangent space is isometric to }\bb{R}^{k+1} \times Y \mbox{, for any  }(Y,\sd_Y) \right\}.
    \end{equation*}
\end{definition}

The next theorem was proved in \cite{Colding1}.
\begin{thm} \label{T2}
    Let $(X, \sd, \bar{x})$ be a non-collapsed Ricci limit space of dimension $n \geq 2$. Then $\mathcal{S}_{n-1} \setminus \mathcal{S}_{n-2} = \emptyset$.
\end{thm}

We next recall some  properties of $\RCD$ spaces and of finite perimeter sets thereof, which will be useful later in the paper. We assume the reader to be familiar with the subject and we refer to the surveys \cite{Asurv, SSurv, GSurv} for an account of the theory. Let us just recall that an $\RCD(K,N)$ space is a metric measure space $(X,\sd,\m)$ with Ricci curvature bounded below by $K\in \mathbb{R}$ and dimension bounded above by $ \, N\in [1,\infty]$ in a synthetic-sense, and whose 
 Sobolev space $W^{1,2}$ is Hilbert. For the present work, it will suffice to consider the finite dimensional case; thus, if not otherwise specified, $N\geq 1$ will be a real number.  Recall that  such  a class is stable under pmGH convergence,  it contains Ricci limit spaces and finite dimensional Alexandrov spaces. 
 
 The next result follows from \cite{MondinoNaber,BScon} (after the seminal works for Ricci limits \cite{Colding1, ChCo3, ColdingNaber}).

\begin{thm}\label{thm:EssDim}
    Let $(X,\sd,\m)$ be an $\RCD(K,N)$ space. Then there exists $k \in \bb{N} \cap [1,N]$, called \emph{essential dimension} of $X$, such that for $\m$-a.e.\;$x \in X$ all tangent spaces in $x$ are isometric to $(\bb{R}^k,\sd_e,0)$.
\end{thm}

The next proposition, from \cite[Theorem $1.1$]{KL}, characterizes the $\RCD$ spaces having essential dimension  $1$.

\begin{proposition} \label{P1}
    Let $(X,\sd,\m)$ be an $\RCD(K,N)$ space of essential dimension equal to $1$. Then $(X, \sd)$ is isometric to a closed connected subset of $\bb{R}$ or to $\mathbb{S}^1(r):=\{x \in \bb{R}^2: |x|=r \}$.
\end{proposition}

We now recall some topological properties of $\RCD$ spaces.
The next proposition is a special case of  \cite[Theorem 5]{Santos}, in the non-collapsing scenario.  We denote by $\mathrm{b}_1(X)$ the first Betti number of $X$.

\begin{proposition} \label{P|uppersemicontinuity betti}
    Let $\big((X_j,\sd_j,\aH^N_j,\bar{x}_j)\big)_{j\in \mathbb{N}}$ be a sequence of pointed  $\RCD(0,N)$ spaces with $\aH^N_j(B_1(\bar{x}_j)) \geq v>0$ converging in pGH-sense to a compact metric space $(X,\sd)$. If $\mathrm{b}_1(X_j) \geq r$ for every $j$, then $\mathrm{b}_1(X) \geq r$. 
\end{proposition}

The next proposition, proved in \cite[Corollary 1.2]{RCDsemilocally} (see also \cite[Theorem 1.2]{wang2021riccilimitspacessemilocally}), establishes the converse inequality in higher generality.

\begin{proposition} \label{P|lowerbetti}
    Let $\big((X_j,\sd_j,\m_j, \bar{x}_j)\big)_{j\in \mathbb{N}}$ be a sequence of pointed $\RCD(K,N)$ spaces converging in $pmGH$-sense to a compact metric measure space $(X,\sd,\m)$. If $\mathrm{b}_1(X_j) \leq r$ for every $j$, then $\mathrm{b}_1(X) \leq r$. 
\end{proposition}

\begin{definition}[Covering space and universal cover] Let $(X, \sd_X )$ be a connected metric space. We say that $(Y, \sd_Y )$ is a \emph{covering space} for $X$ if there exists a continuous map $\pi : Y \to X$
such that for every point $x \in X$ there exists a neighborhood $U_x \subset X$ with the following property:
$\pi^{-1}(U_x)$ is the disjoint union of open subsets of $Y$ each of which is mapped homeomorphically onto
$U_x$ via $\pi$. \\
A connected metric space $(\tilde{X},\sd_{\tilde{X}})$ is said to be a \emph{universal cover} for $(X,\sd_X)$ if it is
a covering space  with the following property: for any other connected covering space $(Y, \sd_Y )$ of
$(X,\sd_X)$, there exists a continuous map $f : \tilde{X} \to Y$ such that the triangle made with the projection
maps onto $X$ commutes. 
\end{definition}

The next result proved in \cite{MondinoWei}  establishes the existence of the universal cover for an $\RCD(K,N)$ space.  

\begin{proposition}
    Let $(X,\sd,\m)$ be an $\RCD(K,N)$ space. Then, it admits a universal cover $(\tilde{X},\sd_{\tilde{X}})$. Moreover there exists a measure $\tilde{\m}$ on $\tilde{X}$ such that $(\tilde{X},\sd_{\tilde{X}}, \tilde{\m})$ is an $\RCD(K,N)$ space.
\end{proposition}

Since $\RCD(K,N)$ spaces are semi-locally simply connected thanks to \cite{RCDsemilocally} the universal is simply connected. The next proposition follows from \cite[Theorem 1.3]{MondinoWei}, after taking into account that the revised fundamental group mentioned in the reference coincides with the usual fundamental group thanks to \cite{RCDsemilocally}.

\begin{proposition} \label{P|splitting universal cover}
    Let $(X,\sd,\m)$ be a compact $\RCD(0,N)$ space with $\mathrm{b}_1(X) \geq k$. Then, its universal cover splits the Euclidean  $\bb{R}^k$ isomomorphically as metric measure space.
\end{proposition}

The next proposition from \cite[Corollary 4.11]{Sormaniwei} shows that the universal cover of a Ricci limit space is a Ricci limit space, as well.

\begin{proposition} \label{P|convergence coverings}
    Let $K \in \bb{R}$ and $n \in \bb{N}$. Let $(M_j^n,g_j, \bar{x}_j)$ be a sequence of Riemannian manifolds with $\Ric_{M_j} \geq K$ converging in pGH-sense to a compact metric space $(X,\sd)$. Let $(\tilde{X},\tilde{\sd})$ be the universal cover of $X$ and fix $\tilde{x} \in \tilde{X}$.    
    Then, there is a sequence of coverings $(\tilde{M}_j,\tilde{g}_j, \tilde{x}_j)$ of $M_j$ converging in pGH-sense to $(\tilde{X},\tilde{\sd},\tilde{x})$.
\end{proposition}

The next proposition follows combining \cite[Theorems $1.1$ and $1.3$]{Anderson} with the fact that each covering of a fixed manifold has volume growth controlled from above by the one  of the universal cover. 

\begin{proposition} \label{P|anderson}
    Let $k,h \in \bb{N} \cup \{0\}$ with $k \leq h$. Let $(M,g,\bar{x})$ be a pointed Riemannian manifold with $\ssf{Vol}(B_t(\bar{x})) \geq c_1 t^k$ for some $c_1>0$ and whose universal cover $(\tilde{M},\tilde{g},\tilde{x})$ satisfies $\widetilde{\ssf{Vol}}(B_t(\tilde{x})) \leq c_2 t^h$ for some $c_2>0$. Then $\mathrm{b}_1(M) \leq h-k$.
\end{proposition}

\section{Integral estimates on scalar curvature and first geometric applications} \label{S3}

\subsection{Proof of Theorem \ref{CT1}} \label{S:proofMain}

Throughtout the section,  $(M^n,g)$ will be a smooth, complete (in some cases compact), $n$-dimensional Riemannian manifold, with $n\geq 2$. We denote by $\R_k:M \to \bb{R}$ the sum of the lowest $k$ eigenvalues, counted with their multiplicity, of the Ricci tensor. We first prove a few lemmas that will be used  to establish Theorem \ref{CT1} from the Introduction.

Given a matrix $A \in \bb{R}^{n \times k}$, we denote by $A^T \in \bb{R}^{k \times n}$ its transpose.

\begin{lemma} \label{L14}
    Let $k,n \in \bb{N}$ with $k \leq n$, and let $c,\epsilon>0$ be fixed. There exist $c_1(n,\epsilon),\epsilon'(n,\epsilon)>0$ arbitrarily small satisfying the following. Let $A=(a_{j,i}) \in \bb{R}^{n \times k}$ be a matrix with $|A^T A-I_k| \leq c_1$ and $|A|^2 \leq c$. Consider real numbers $-\epsilon' \leq \lambda_1 \leq \cdots \leq \lambda_n$ with $\lambda_k \geq \epsilon/2k$.
    Then
    \[
    0 \vee
   \sum_{i=1}^k\sum_{j=1}^n \lambda_j a_{j,i}^2 \geq \sqrt{\epsilon'} (\lambda_1+ \cdots + \lambda_k).
    \]
    \begin{proof}
        Assume by contradiction that the statement fails. 
        Hence, we find sequences $c_h,\epsilon'_h \downarrow 0$ and 
        $A_h=(a^h_{j,i}) \in \bb{R}^{n \times k}$ with $|A^T_h A_h-I_k| \leq c_h$, $|A_h|^2 \leq c$ and real numbers $-\epsilon'_h \leq \lambda^h_1 \leq \cdots \leq \lambda^h_n$ with $\lambda_k \geq \epsilon/2k$, with
    \[
    0 \vee
    \sum_{i=1}^k\sum_{j=1}^n \lambda^h_j (a^h_{j,i})^2 < \sqrt{\epsilon'_h} (\lambda^h_1+ \cdots + \lambda^h_k),\quad  \text{for all } h\in \mathbb{N}.
    \]
    In particular, dividing both sides by $\lambda_k^h$, it holds 
    \begin{equation} \label{E33}
    \sum_{i=1}^k
   \sum_{j=1}^{k-1} \frac{\lambda^h_j}{\lambda^h_k} (a^h_{j,i})^2
   +
   \sum_{i=1}^k\sum_{j=k}^n (a^h_{j,i})^2
   \leq 
   k\sqrt{\epsilon'_h},
    \end{equation}
    where the first summand in the l.h.s. is to be considered zero if $k=1$.
\par 
Moreover, since $|A_h|^2 \leq c$, $-\epsilon'_h \leq \lambda^h_1 \leq \cdots \leq \lambda^h_n$ and $\lambda^h_k \geq \epsilon/2k$, it holds
    \[
    \sum_{i=1}^k
   \sum_{j=1}^{k-1} \frac{\lambda^h_j}{\lambda^h_k} (a^h_{j,i})^2 \geq -2k^2c\frac{\epsilon'_h}{\epsilon}.
    \]
    Combining with \eqref{E33}, we obtain
    \[
   \sum_{i=1}^k\sum_{j=k}^n (a^h_{j,i})^2
   \leq 
   k\sqrt{\epsilon'_h}+2k^2c\frac{\epsilon'_h}{\epsilon}.
    \]
    Taking the limit as $h \to \infty$, we obtain a matrix $A=(a_{j,i}) \in \bb{R}^{n \times k}$ such that 
    \begin{equation*} %\label{E23}
    A^TA=I_k \quad \text{and} \quad  \sum_{i=1}^k\sum_{j=k}^n (a_{j,i})^2 = 0.
    \end{equation*}
    The last two identities are clearly in contradiction: indeed, the former states that the  columns of $A$ are $k$ orthonormal vectors in $\bb{R}^n$, while the latter  implies that their  span is at most $k-1$ dimensional.
    \end{proof}
\end{lemma}

The next lemma will be key to prove Theorem \ref{CT1}.

\begin{lemma} \label{L2}
    Let $k,n \in \bb{N}$, with   $k \leq n$, and let $s \in (0,1)$ be fixed. 
    For every $\epsilon >0$ there exists $\delta (n,\epsilon,s) >0$ such that if $(M^n,g)$ has $\Ric_{M} \geq - \delta$ on $B_{20}(x)$, and $B_{10}(x) \subset M$ is
     $(k,\delta)$-symmetric, then:
    \begin{itemize}
\item    There exists $E \subset B_1(x)$ such that
        \begin{equation}\label{eq:IntERk<eps}
         \int_{E} \R_k \,\de\ssf{Vol} \leq \epsilon \ssf{Vol}(B_1(x));
         \end{equation}
         \item There exist $x_j \in B_1(x)$ and $r_j \in [0,1/2]$, with  $j\in \mathbb{N}$, such that
         \begin{equation}\label{eq:B1-EBalls}
        B_1(x) \setminus E \subset \bigcup_{j \in \bb{N}} B_{r_j}(x_j) \text{ with } 
        \sum_{j \in \bb{N}} \ssf{Vol}(B_{r_j}(x_j))r_j^{-2s} \leq \epsilon \ssf{Vol}(B_1(x)).
        \end{equation}
         \end{itemize}
        In particular, $\ssf{Vol}(B_1(x) \setminus E) \leq  \epsilon \ssf{Vol}(B_1(x))$ and
        $\fint_{E} \R_k \,\de\ssf{Vol}\leq \frac{\epsilon}{1-\epsilon}$.
        \end{lemma}

        \begin{proof}
            Let $\epsilon>0$ be fixed. Let $0<\epsilon'(n,\epsilon)<\epsilon^2$ be given by Lemma \ref{L14} and let  $0<\delta<\epsilon'$ be small enough so that if $B_{10}(x) \subset M$ is $(k,\delta)$-symmetric, there exists a harmonic $(k,\epsilon')$-splitting map $u:B_5(x) \to \bb{R}^k$. 
            Plugging the $j$-th entry $u_j$ of the map $u$ into the Bochner Formula yields
            \begin{equation} \label{E3}
            \Delta \frac{|\nabla u_j|^2-1}{2}=|\Hess (u_j)|^2+\Ric(\nabla u_j,\nabla u_j), \quad \text{on $B_5(x)$}.
            \end{equation}
            Theorem \ref{TGoodCut} ensures the existence of a  non-negative function  $\phi \in C^\infty_c(B_2(x))$, with $\phi\equiv 1$ on $B_1(x)$ and such that 
            \[
            |\Delta \phi|+|\nabla \phi| \leq c(n).
            \]
            Multiplying \eqref{E3} by $\phi$, integrating, and recalling Definition \ref{def:delta-split-map}, we obtain
            \[
            \fint_{B_2(x)} \phi \Ric(\nabla u_j,\nabla u_j)  \,\de\ssf{Vol} \leq c(n) \epsilon'.
            \]
            Since $\Ric_{M} \geq -\delta \geq -\epsilon'$ and $\sup_{B_5(x)} |\nabla u| \leq 2$, it holds
            \[
            \fint_{B_1(x)}  \Ric(\nabla u_j,\nabla u_j)  \,\de\ssf{Vol} \leq
            c(n)\fint_{B_2(x)} \phi \Ric(\nabla u_j,\nabla u_j)  \,\de\ssf{Vol}  + c(n)\epsilon'
            \leq c(n) \epsilon'.
            \]
            Summing over $j \in \{1, ..., k\}$, we get
            \begin{equation} \label{E1}
            \fint_{B_1(x)} \sum_{j=1}^k \Ric(\nabla u_j,\nabla u_j)  \,\de\ssf{Vol} \leq c(n) \epsilon'.
            \end{equation}
            \par
            We now define $E \subset B_1(x)$ to be the set of points where $\{\nabla u_j\}_{j=1}^k$ are almost orthonormal. For each point $p \in M$, we fix an orthonormal basis so that $T_pM \equiv \bb{R}^n$ and we use the convention that $\nabla u \in \bb{R}^{n \times k}$ (i.e., each $\nabla u_j$ is a column of $\nabla u$). Let $c_1(n,\epsilon)>0$ be the constant given by Lemma \ref{L14}, and set
            \[
            E:=\{y \in B_1(x): | \nabla u^T \nabla u-I_k| \leq c_1\}.
            \]
            Observe that the condition in the previous definition is independent of the chosen orthonormal basis for $T_pM$. Fix now an orthonormal basis in $T_pM$ such that $\Ric_M$ in $T_pM$ is represented by the diagonal matrix $\Lambda \in \bb{R}^{n \times n}$, whose eigenvalues are $-\epsilon' \leq \lambda_1 \leq \cdots \leq \lambda_n$.

            Consider now the subsets of $E$ defined by
            \[
            E_+:=E \cap \{\R_k \geq \epsilon/2\} \quad \text{and} \quad 
            E_-:=E \cap \{\R_k < \epsilon/2\}.
            \]
            At every point of $E_+$, we have $\lambda_k \geq \epsilon/2k$. Hence, by Lemma \ref{L14}, it holds
            \begin{align*}
            0 & \vee \sum_{j=1}^k \Ric(\nabla u_j,\nabla u_j)=0 \vee \sum_{j=1 }^k \nabla u_j^T \Lambda \nabla u_j
            \\ 
            & = 0 \vee \sum_{j=1}^k \sum_{i=1}^n \lambda_i (\nabla u_{j,i})^2 \geq \sqrt{\epsilon'} (\lambda_1 + \dots + \lambda_k) =\sqrt{\epsilon'} \R_k.
            \end{align*}
            Combining with \eqref{E1}, we obtain
            \begin{align} 
            \label{E2}
            \nonumber 
            \int_{E} & \R_k \,\de\ssf{Vol} = 
            \int_{E_-} \R_k \,\de\ssf{Vol}+
            \int_{E_+} \R_k \,\de\ssf{Vol} 
            \\ 
            & \leq 
            (\epsilon/2) \ssf{Vol}(B_1(x))+\int_{E_+} \R_k \,\de\ssf{Vol} 
            \leq (c(n)\sqrt{\epsilon'}+\epsilon/2) \ssf{Vol}(B_1(x)).
            \end{align}
            Choosing $\epsilon'>0$ small enough, we obtain \eqref{eq:IntERk<eps}.
            
            To conclude the proof, we need to show that $B_1(x) \setminus E$ is contained in a union of sufficiently small balls. Let $\eta(\epsilon,n,s)>0$ to be chosen later and consider the set $E' \subset B_1(x)$ given by
            \[
            E':=\Big\{y \in B_1(x): r^{2s} \fint_{B_r(y)} |\Hess(u)|^2 \,\de\ssf{Vol} \leq \eta\, , \; \text{ for every } r \in (0,1/10) \Big\}.
            \]
            We claim that
            \begin{equation}\label{eq:claimE'}
            \begin{split}
             E' &\subset E \quad\text{ and} \\
            B_1(x) \setminus E' \subset \bigcup_{i \in \bb{N}} B_{r_j}(x_j) &\text{ with } 
        \sum_{i \in \bb{N}} \ssf{Vol}(B_{r_j}(x_j)) r_j^{-2s} \leq \epsilon \ssf{Vol}(B_1(x)).
        \end{split}
            \end{equation}
            The thesis \eqref{eq:B1-EBalls} will follow by combining the claim \eqref{eq:claimE'} with  \eqref{E2}.
            \par 
            We first prove that $E' \subset E$. It is sufficient to show that at every point of $E'$ the quantity $|\nabla u^T  \nabla u-I_k|$ is smaller than the constant $c_1(\epsilon,n)$ appearing in the definition of $E$. We prove it by a telescopic argument (cf.\;\cite{MonBruSem}).
            Set $\mathcal{E}:=\sum_{i,j}| \nabla u_j \cdot \nabla u_j-\delta_{ij}| $.
            For every $j \in \bb{N}$ and $y \in B_1(x)$, we have
            \begin{align*}
            \fint_{B_{2^{-j}}(y)} \mathcal{E} \,\de\ssf{Vol}
            & = \Big| \fint_{B_1(y)} \mathcal{E} \,\de\ssf{Vol}
            -\fint_{B_1(y)} \mathcal{E} \,\de\ssf{Vol}
            +
            \fint_{B_{1/2}(y)} \mathcal{E} \,\de\ssf{Vol}
            - ...
            + \fint_{B_{2^{-j}}(y)} \mathcal{E} \,\de\ssf{Vol} \Big|
            \\
            & \leq c(n)\epsilon' + \sum_{j \in \bb{N}} \Big|
            \fint_{B_{2^{-j+1}}(y)} \mathcal{E} \,\de\ssf{Vol}
            - \fint_{B_{2^{-j}}(y)} \mathcal{E} \,\de\ssf{Vol}
            \Big|
            \\
            & \leq c(n) \epsilon' + c(n) \sum_{j \in \bb{N}} 2^{-j}  \fint_{B_{2^{-j}}(y)} |\Hess(u)| \,\de\ssf{Vol},
            \end{align*}
            where in the last inequality we used the Poincar\'e inequality.
            Using Hölder inequality  and the definition of $E'$, we infer that
            \begin{align*}
            \fint_{B_{2^{-j}}(y)} \mathcal{E} \,\de\ssf{Vol} & \leq
            c(n) \epsilon' + c(n) \sum_{j \in \bb{N}} 2^{-j(1-s)} 
            \Big( 2^{-2sj} \fint_{B_{2^{-j}}(y)} |\Hess(u)|^2 \,\de\ssf{Vol} \Big)^{1/2} \\ 
            & \leq c(n) \epsilon' + c(n,s) \sqrt{\eta}.
            \end{align*}
            Choosing $\eta$ and $\epsilon'$ to be sufficiently small, and letting $j$ go to $+ \infty$, we conclude that $E' \subset E$.
            
            From now until the end of the proof, $\eta>0$ is fixed by the previous step. To complete the proof, we need to show that  
            \[
            B_1(x) \setminus E' \subset \bigcup_{i \in \bb{N}} B_{r_j}(x_j) 
            \]
            with
            \[
        \sum_{i \in \bb{N}} \ssf{Vol}(B_{r_j}(x_j)) r_j^{-2s} \leq \epsilon \ssf{Vol}(B_1(x)), \;  \text{ and } r_j \in [0,1/2].
        \]
        To prove this, we use a Vitali covering argument coupled with a scale-picking (cf.\;\cite{MonBruSem}). For every $y \in B_1(x)$ such that
        \[
        \sup_{0<r<1/10}r^{2s} \fint_{B_r(y)} |\Hess(u)|^2 \,\de\ssf{Vol} > \eta,
        \]
        we set $r_y \in [0,1/10]$ to be the maximal radius such that
        \[
        r_y^{2s}
        \fint_{B_{r_y}(y)} |\Hess(u)|^2 \,\de\ssf{Vol} \geq \eta.
        \]
        Clearly,
        \[
        \int_{B_{r_y}(y)} |\Hess(u)|^2 \,\de\ssf{Vol} \geq  \ssf{Vol}(B_{r_y}(y)) \eta r_y^{-2s}.
        \]
        By  Vitali covering theorem, we can cover $E'$ with a countable collection of balls $\{B_{5r_{y_j}}(y_j)\}_{j \in \bb{N}}$ with $r_{y_j} \in (0,1/10]$, such that $\{B_{r_{y_j}}(y_j)\}_{j \in \bb{N}}$ are disjoint. Hence
        \begin{align*}
        \sum_{j \in \bb{N}} \frac{\ssf{Vol}(B_{5r_{y_j}}(y_j))} {(5r_{y_j})^{2s}} 
        & \leq \frac{c(n)}{\eta}  \sum_{j \in \bb{N}} \int_{B_{r_{y_j}}(y_j)} |\Hess(u)|^2 \,\de\ssf{Vol} \\
        & \leq \frac{c(n)}{\eta} \int_{B_2(x)} |\Hess(u)|^2 \,\de\ssf{Vol} \leq  \frac{c(n)}{\eta} \ssf{Vol}(B_1(x))\epsilon'.
        \end{align*}
        Observe that, by construction, $5r_j \in [0,1/2]$ for every $j \in \bb{N}$. Since $\eta>0$ is fixed, we can choose $\epsilon'=\epsilon \eta c(n)^{-1}>0$ to conclude the proof. 
        \end{proof}

\begin{corollary} \label{C1}
    Let $k,n \in \bb{N}$ with $k \leq n$ and $s \in (0,1)$ be fixed. 
    For every $\epsilon >0$ there exists $\delta>0$ such that if $(M^n,g)$ has $\Ric_M \geq -\delta$ and $B_{10}(x) \subset M$ is $(k,\delta)$-symmetric, then there exists $E \subset B_1(x)$ such that
        \[
         \int_{E} |\R_k|^s \,\de\ssf{Vol} \leq \epsilon \ssf{Vol}(B_1(x))
         \]
         with
         \[
         \begin{split}
        B_1(x) \setminus E \subset \bigcup_{i \in \bb{N}} B_{r_j}(x_j)& \text{ and } \\ 
        \sum_{j \in \bb{N}} \ssf{Vol}(B_{r_j}(x_j))r_j^{-2s} \leq \epsilon \ssf{Vol}(B_1(x)),& \quad r_j \in [0,1/2]  \;\text{ for all }\; j \in \bb{N}.
        \end{split}
        \]
        \begin{proof}
            Let $0<\delta<\epsilon_1<\epsilon$ be sufficiently small and let $E$ be the set given by Lemma \ref{L2} relative to $\epsilon_1$.
            Since by assumption $\Ric_M \geq -\delta$, then  clearly
            \begin{equation}\label{eq:Rkgeq-kdelta}
                \R_k \geq -k \delta.
            \end{equation}
            By Hölder inequality applied to the conjugate exponents $p=s^{-1}>1$ and  $q=(1-s)^{-1}>1$, using    Lemma \ref{L2} and \eqref{eq:Rkgeq-kdelta}, we obtain
            \[
             \int_{E} (\R_k \vee 0)^s \,\de\ssf{Vol} \leq \Big( \int_E \R_k \vee 0 \,\de\ssf{Vol} \Big)^s \ssf{Vol}(B_1(x))^{1-s}  \leq \epsilon_1 \ssf{Vol}(B_1(x)). 
            \]
            Hence,
            \[
            \int_{E} |\R_k|^s \,\de\ssf{Vol} \leq 
             \int_{E} (\R_k \vee 0)^s \,\de\ssf{Vol}+ (k \delta)^s \ssf{Vol}(B_1(x)) \leq 
            (\epsilon_1+(k \delta)^s) \ssf{Vol}(B_1(x)).
            \]
            We conclude by choosing $\epsilon_1,\delta>0$ such that $\epsilon_1+(k \delta)^s \leq \epsilon$.
        \end{proof}
\end{corollary}

We now prove Theorem \ref{CT1} from the Introduction.

\begin{thm} 
    Let $L,k,n \in \bb{N}$ with $k \leq n$ and $s \in (0,1)$ be fixed. Then 
    \begin{align}
    &\sup \Big\{\fint_{B_1(x)} \R_{k} \wedge L \, \de \ssf{Vol} : (M,g,x) \text{ has } \Ric_{M} \geq -\delta, \text{ and } B_{10}(x) \text{ is  $(k,\delta)$-symmetric} 
    \Big\} ,
    \\
    &  \sup \Big\{\fint_{B_1(x)} |\R|^s \, \de \ssf{Vol} : (M,g,x) \text{ has } \Ric_{M} \geq -\delta, \text{ and } B_{10}(x) \text{ is } \delta \text{-regular} 
    \Big\}  \label{eq:supintRs}
    \end{align}
    converge to $0$ as $\delta \searrow 0$. 
    \begin{proof}
    Let $k,n \in \bb{N}$ with $k < n$.  We show that for every $0<\epsilon <1/2$ there exists $\delta>0$ such that if $(M^n,g)$ has $\Ric_M \geq -\delta$ and $B_{10}(x) \subset M$ is $(k,\delta)$-symmetric, then
    \begin{equation} \label{E5}
    \fint_{B_1(x)} \R_k \wedge L \,\de\ssf{Vol} \leq  \epsilon.
    \end{equation}
    This would imply the first part of the statement.
     Fix $\epsilon_1(\epsilon,n)>0$ -- to be chosen later -- and let $0<\delta(\epsilon_1)<\epsilon_1$ be given by Lemma \ref{L2} so that there exists $E \subset B_1(x)$ such that
        \[
         \int_{E} \R_k \,\de\ssf{Vol} \leq \epsilon_1 \ssf{Vol}(B_1(x)),
        \]
        and $\ssf{Vol}(B_1(x) \setminus E) \leq \epsilon_1 \ssf{Vol}(B_1(x))$.
    Hence,
    \begin{align*}
    \int_{B_1(x)} \R_k \wedge L \,\de\ssf{Vol}  \leq \int_{E} \R_k  \,\de\ssf{Vol}+\int_{B_1(x) \setminus E} 0 \vee \R_k \wedge L \,\de\ssf{Vol} \\
    \leq 
    (\epsilon_1+L\epsilon_1 )\ssf{Vol}(B_1(x)),
    \end{align*}
    which implies \eqref{E5} if $\epsilon_1$ is chosen small enough.
\\
We next show \eqref{eq:supintRs}. To this aim, fix $s \in (0,1)$ and consider the case  $k=n$.
    We want show that for every $0<\epsilon <1/2$ there exists $\delta>0$ such that if $(M^n,g,x)$ has $\Ric_M \geq -\delta$ and $B_{10}(x) \subset M$ is $\delta$-regular, then
    \[
    \fint_{B_1(x)} |\R|^s \,\de\ssf{Vol} \leq  \epsilon.
    \]
    Let $\delta'>0$ be as in Lemma \ref{L2} relative to $\epsilon$.
        Let $0<\delta<\delta'$ be small enough such that if $(M^n,g,x)$ has $\Ric_M \geq -\delta$ and $B_{10}(x) \subset M$ is $\delta$-regular, then $B_r(y)$ is $\delta'$-regular for every $r \in (0,5)$ and every $y \in B_5(x)$. This is possible thanks to Theorem \ref{T1}. 
        Define the measure $$\mu:=|\R|^s \,\de\ssf{Vol} \mres B_1(x).$$ Consider the set $E_1 \subset B_1(x)$ given by Corollary \ref{C1} such that
        \[
        \mu(E_1) \leq \epsilon \ssf{Vol}(B_1(x)),  \quad  B_1(x) \setminus E_1 \subset \bigcup_{i \in \bb{N}} B_{r_{i}}(x_{i}) 
        \]
        with 
        \[
        \sum_{i \in \bb{N}} \ssf{Vol}(B_{r_{i}}(x_{i}))r_{i}^{-2s} \leq \epsilon \ssf{Vol}(B_1(x)) \text{ and } r_i \in [0,1/2]. 
        \]
        In particular, it holds
        \[
        \ssf{Vol}(B_1(x) \setminus E_1) \leq 
        % v(n) \sum_{i \in \bb{N}} r_{i}^{n} \leq 
        % v(n) 
        \epsilon \ssf{Vol}(B_1(x)).
        \]
        In each ball $B_{r_{i}}(x_{i})$, we can apply the scale invariant version of Corollary \ref{C1} to obtain a set $E_{2,i} \subset B_{r_{i}}(x_{i})$ such that
        \[
        \mu(E_{2,i}) \leq \epsilon  r_{i}^{-2s}\ssf{Vol}(B_{r_i}(x_i)), \quad  B_{r_{i}}(x_{i}) \setminus E_{2,i} \subset \bigcup_{j \in \bb{N}} B_{r_{i,j}}(x_{i,j})
        \]
        with
        \[ 
        \sum_{j \in \bb{N}} \ssf{Vol}(B_{r_{i,j}}(x_{i,j}))r_{i,j}^{-2s} \leq \epsilon  r_i^{-2s}\ssf{Vol}(B_{r_i}(x_i)) \text{ and } r_{i,j} \in [0,r_i/2]. 
        \]
        Defining 
        \[
        E_2 := E_1 \cup \bigcup_{i \in \bb{N}} E_{2,i},
        \]
        we obtain
        \begin{align*}
        \mu(E_2) & \leq \epsilon  \ssf{Vol}(B_1(x)) + \sum_{i \in \bb{N}} \mu (E_{2,i}) \\
        & \leq 
         \epsilon \ssf{Vol}(B_1(x))  + \epsilon  \sum_{i \in \bb{N}} r_{i}^{-2s} \ssf{Vol}(B_{r_i}(x_i))
         \leq (\epsilon +\epsilon ^2)\ssf{Vol}(B_1(x)) .
        \end{align*}
        Moreover, it holds
        \[
        \ssf{Vol}(B_1(x) \setminus E_2) \leq \sum_{i,j \in \bb{N}} \ssf{Vol}(B_{r_{i,j}}(x_{i,j}))
         \leq  {\epsilon}^2 \ssf{Vol}(B_1(x)).
        \]
        If we keep iterating this argument, we obtain increasing sets $E_k$ such that $$\mu(E_k) \leq (\epsilon + ... + \epsilon^k)\;\ssf{Vol}(B_1(x)).$$ Denoting by 
        $$E=\bigcup_{k\in \mathbb{N}} E_k,$$ we obtain that $$\mu(E) \leq c \epsilon \ssf{Vol}(B_1(x)),$$ where $c>0$ is a universal constant.  To conclude the proof, it suffices to show that 
        $$\mu(B_1(x)\setminus E)=0.$$ 
        To this aim, observe that $\ssf{Vol}(B_1(x)\setminus E)=0$,  since 
        $
        \ssf{Vol}(B_1(x) \setminus E_k) \leq {\epsilon}^k \ssf{Vol}(B_1(x)).
        $
\end{proof}
\end{thm}

\begin{remark} \label{R5}
    Inspecting the proof of Lemma \ref{L2}, one realizes that a priori $L^p$ estimates on harmonic splitting maps for $p>2$ would imply that
    \[
    \sup \Big\{\fint_{B_1(x)} |\R| \, \de \ssf{Vol} : (M,g,x) \text{ has } \Ric_{M} \geq -\delta \text{ and } B_{10}(x) \text{ is } \delta \text{-regular }
    \Big\} \to 0
    \]
    as $\delta \to 0$.
    This would allow to obtain a priori integral bounds on the scalar curvature outside of the quantitative singular set of a manifold with Ricci curvature bounded from below.
    Unfortunately, it is shown in \cite{DephilippisHarmonic} that such estimates cannot exist in such generality.
\end{remark}

\subsection{First geometric  applications in almost non-negative Ricci and integrally-positive scalar curvature} \label{S4}

We now obtain the first geometric applications of the results of the previous section. We introduce some notation first (cf. \cite{ShenKricci}).

\begin{definition}[$k^{th}$ Ricci curvature bounds]
    Let $(M^n,g)$ be a manifold, let $1 \leq k \leq n-1$, let $x \in M$, and let $c \in \bb{R}$. We denote by $R_M(\cdot,\cdot,\cdot,\cdot)$ the Riemann curvature tensor of $M$. We say that $\Ric_{M,k} \geq c$ (resp. $\leq c$) in $T_xM$ if, for all $(k+1)$-dimensional subspaces $V \subset T_xM$, it holds
\[
\sum_{i=1}^{k+1} R_M(e_i,v,v,e_i) \geq c \, 
 \,(\text{resp.} \leq c)
\]
for every $v \in V$ and every orthonormal basis $\{e_1,\dots,e_{k+1}\}$ of $V$.
\end{definition}

One can check that:
\begin{enumerate}
    \item $\Ric_{M,k} \geq c$ implies $\R_k \geq c$.
    \item $\Ric_{M,k} \geq c$ implies $\Ric_{M,k+1} \geq \frac{k}{k-1} c$.
    \item $\Ric_{M,n-1} \geq c$ if and only if $\Ric_{M} \geq c$.
    \item $\Ric_{1,M} \geq c$ if and only if $\ssf{Sec}_{M} \geq c$.
\end{enumerate}

\begin{definition}[Integral $k^{th}$ Ricci curvature bounds] \label{D5}
Let $(M^n,g)$ be a manifold, let $1 \leq k \leq n-1$, let $A \subset M$ be measurable, let $c \in \bb{R}$, and let $F:\bb{R} \to \bb{R}$ be a measurable function.
We say that
\[
\int_{A} F(\Ric_{M,k}) \, d\ssf{Vol} \geq c
\]
if there exists a measurable function $f:A \to \bb{R}$ such that:
\begin{enumerate}
    \item $
    \int_A f \, d\ssf{Vol} \geq c.
    $
    \item For all $x \in A$ and $(k+1)$-dimensional subspaces $V \subset T_xM$, it holds
\[
F\Big( \sum_{i=1}^{k+1} R_M(e_i,v,v,e_i) \Big) \geq f(x)
\]
for every $v \in V$ and every orthonormal basis $\{e_1,\dots,e_{k+1}\}$ of $V$.
\end{enumerate}
We say that $
\int_{A} F(\ssf{Sec}_M ) \, d\ssf{Vol} \geq c
$ if $
\int_{A}  F(\Ric_{1,M}) \, d\ssf{Vol} \geq c
$.
\end{definition}

The next theorem should be compared with \cite[Theorem 1.1]{wang2023positive}, where similar conclusions were obtained under stronger assumptions: non-negative Ricci curvature (below, almost non-negative Ricci curvature) and scalar curvature greater than 2  (below, a local integral lower bound on the scalar curvature). 
Whenever we consider a sequence of manifolds $(M_j,g_j)$, we write $\R_{M_j,k}$ instead of $\R_k$ for the sake of clarity.

\begin{thm} \label{CT2}
    Let $s \in (0,1)$, let $\epsilon,v,L \in (0,+\infty)$ and let $k,n \in \bb{N}$ with $k \leq n$. 
    \begin{enumerate}
        \item \label{Thm4.7-1} There exists $\delta>0$ such that the following holds. Let $(M^n,g,p)$ be a manifold with $\Ric_{M} \geq -\delta$ on $M$ and
    \[
     \fint_{B_1(p)} |\R_{M}|^s \, \de \ssf{Vol}  \geq \epsilon, \quad \ssf{Vol}(B_1(p)) \geq v.
    \] 
    Then $\sd_{\rm{pGH}}(M,Y) \geq \delta$, for any metric space $(Y,\sd_y,y)$  splitting $\bb{R}^{n-1}$ isometrically.
    \item \label{Thm4.7-2}
    There exists $\delta>0$ such that the following holds. Let $(M^n,g,p)$ be a manifold with $\Ric_{M} \geq -\delta$ on $M$, and
    \[
    \fint_{B_1(p)} \ssf{Ric}_{M,n-2} \wedge 0 \, \de \ssf{Vol}  \geq -\delta, \quad 
     \fint_{B_1(p)} \R_{M} \wedge L \, \de \ssf{Vol}  \geq \epsilon.
    \] 
     Then $\sd_{\rm{pGH}}(M,Y) \geq \delta$, for any metric space $(Y,\sd_y,y)$  splitting $\bb{R}^{n-1}$ isometrically.
    \end{enumerate}
\end{thm}
    
\begin{proof} 
\textbf{Proof of \ref{Thm4.7-1}.}
Assume by contradiction that there exist $\epsilon_0>0$, $v_0>0$, $\delta_j \downarrow 0$ and pointed Riemannian manifolds $(M_j,g_j,p_j)$ with $\Ric_{M_j} \geq -\delta_j$, $\ssf{Vol}_j(B_1(p_j)) \geq v_0$, and
\[
\fint_{B_1(p_j)}|\R_{M_j}|^s \, \de \ssf{Vol}_j \geq \epsilon_0,
\]
such that
$M_j \to X$ in  pGH-sense, where $(X,\sd,p)$ splits $\bb{R}^{n-1}$ isometrically.
Thanks to Proposition \ref{P1} and Theorem \ref{T2}, then either $X=\bb{R}^n$ or $X=\bb{R}^{n-1} \times S^1_r$, for some $r>0$. 
\smallskip

\textbf{Case $X=\bb{R}^n$}.
Applying the second assertion of Theorem \ref{CT1} to $(M_j, g_j, p_j)$, for $j$ large enough, we reach a contradiction.

\smallskip
\textbf{Case $X=\bb{R}^{n-1} \times S^1_r$}.  
Let $s \in (0,1)$. We claim that for every $j$ there exists $p_j^s \in B_1(p_j)$ such that
\begin{equation} \label{eqcorretta}
\fint_{B_s(p_j^s)} |\R_{M_j}|^s \, d\ssf{Vol}_j \geq c(n,s)\epsilon_0.
\end{equation}
If \eqref{eqcorretta} holds, we reach a contradiction by applying the second assertion of Theorem \ref{CT1} to the rescaled sequence $(M_j, s^{-2} g_j, p_j)$, for $s \in (0,1)$ small enough independent of $j$.

We prove \eqref{eqcorretta}. Let $j$ be fixed. Let $\{q_i\}_{i=1}^l \subset B_1(p_j)$ be such that $\{B_s(q_i)\}_{i=1}^l$ covers $B_1(p_j)$, while the balls $\{B_{s/5}(q_i)\}_{i=1}^l$ are disjoint. Assume without loss of generality that
\[
\int_{B_s(q_1)}|\R_{M_j}|^s \, d\ssf{Vol}_j=\max_{i}\int_{B_s(q_i)}|\R_{M_j}|^s \, d\ssf{Vol}_j.
\]
Then,
\[
\int_{B_1(p_j)}|\R_{M_j}|^s \, d\ssf{Vol}_j \leq l \int_{B_s(q_1)}|\R_{M_j}|^s \, d\ssf{Vol}_j.
\]
Hence, to conclude, it suffices to show that $l \leq c(n,s)\frac{\ssf{Vol}_j(B_1(p_j))}{\ssf{Vol}_j(B_s(q_1))}$.
To this aim, note that
\begin{align*}
\ssf{Vol}_j(B_1(p_j))  \geq c(n)\ssf{Vol}_j(B_2(p_j)) \geq c(n)\sum_{i=1}^l \ssf{Vol}_j(B_{s/5}(q_i)) \\
 \geq c(n,s)l\ssf{Vol}_j(B_s(q_1)).
\end{align*}
This concludes the proof of item \ref{Thm4.7-1}.
\smallskip

\textbf{Proof of \ref{Thm4.7-2}.}
By contradiction, assume there exist numbers $\delta_j \downarrow 0$ and pointed Riemannian manifolds $(M_j,g_j,p_j)$ with $\Ric_{M_j} \geq -\delta_j$ on $M_j$, and
\begin{equation} \label{E4}
\fint_{B_1(p_j)} \ssf{Ric}_{M_j,n-2} \wedge 0 \, \de \ssf{Vol}_j  \geq -\delta_j, \quad 
\fint_{B_1(p_j)}\R_{M_j} \wedge L \, \de \ssf{Vol}_j \geq \epsilon,
\end{equation}
such that
$M_j \to X$ in pGH-sense, where $(X,\sd,p)$ splits $\bb{R}^{n-1}$ isometrically. 

By the first assertion of Theorem \ref{CT1}, it holds
\begin{equation}\label{eq:Rjn-1to0}
\fint_{B_1(p_j)} \R_{M_j,n-1} \wedge L \,\de\ssf{Vol}_j \to 0, \quad \text{ as } j\to \infty.
\end{equation}
Consider an orthonormal base $\{e_1,\cdots,e_n\}$ for the tangent space of the manifold $M_j$ in a point $m \in M_j$, and assume that $\Ric_{M_j}$ in this basis is represented by a diagonal matrix $A=(a_{h,l}) \in \bb{R}^{n \times n}$ such that $a_{h,h} \leq a_{h+1,h+1}$ for every $h \in \{1, \cdots ,n-1\}$. It holds
\begin{align*}
\R_{M_j} = \sum_{h \neq l} \ssf{Sec}_{M_j}(e_h,e_l)&=\R_{M_j,n-1}+\sum_{h=1}^{n-1} \ssf{Sec}_{M_j}(e_n,e_h) \\
&=2\R_{M_j,n-1}-\sum_{\substack{ h \neq n \\ l \neq n}}\ssf{Sec}_{M_j}(e_h,e_l).
\end{align*}
Hence,
\begin{equation} \label{eqff}
\R_{M_j} \wedge L \leq 2(0 \vee \R_{M_j,n-1} \wedge L)-\sum_{h=1}^{n-1}\Big(\sum_{\substack{ l \neq h \\ l \neq n}}\ssf{Sec}_{M_j} (e_h,e_l)\Big) \wedge 0. 
% \\
% & \leq  2(\R_{M_j,n-1} \wedge L)+2n\delta_j-n\ssf{Ric}_{M_j,n-2} \wedge 0.
\end{equation}
By the first inequality of \eqref{E4}, for every $h \in \{1,\cdots,n-1\}$, it follows
\[
\fint_{B_1(p_j)}\Big(\sum_{\substack{ l \neq h \\ l \neq n}}\ssf{Sec}_{M_j} (e_h,e_l)\Big) \wedge 0 \, d\ssf{Vol} \geq -\delta_j.
\]
Combining with \eqref{eqff}, it holds
\[
\fint_{B_1(p_j)} \R_{M_j} \wedge L \,\de\ssf{Vol}_j \leq 2\fint_{B_1(p_j)} \R_{M_j,n-1} \wedge L \,\de\ssf{Vol}_j + c(n)\delta_j.
\]
From \eqref{eq:Rjn-1to0}, we infer that
\[
\limsup_{j \to +\infty} \fint_{B_1(p_j)} \R_{M_j} \wedge L \,\de\ssf{Vol}_j= 0,
\]
contradicting the second inequality of \eqref{E4}.
\end{proof}

We now establish some first results of  ``dimension drop" in general dimension. Building on top of Section \ref{S1}, in Section \ref{Sec:MainGeom}   we will obtain stronger conclusions in dimension $3$, or in general dimensions under positivity conditions (in an integral sense) on  $\R_k$, $k\leq 2$. 

We recall that if a pointed non-compact manifold $(M,g,p)$ has non-negative Ricci curvature, then Gromov pre-compactness theorem ensures that, for every sequence $(M,g/r_j,p)$ as $r_j \to + \infty$, there exists a subsequence converging in pGH sense to a metric space $(X,\sd,p_\infty)$. 
Such  spaces obtained via blow-downs are called  \emph{tangent cones at infinity} of $M$.
We now prove Theorem \ref{thmInt3} from the Introduction.

\begin{thm} \label{T|tangentcones}
    Let $(M^n,g,p)$ be a Riemannian manifold; let $L>0$ and $\alpha \in (0,2)$.
    \begin{enumerate}
        \item \label{item1blowdown}
        Assume that $\Ric \geq 0$ and
    \[ 
    \lim_{r \to + \infty} \fint_{B_r(p)} 0 \vee (r^{2-\alpha}\R_k) \wedge L \,\de\ssf{Vol}=L.
    \]
    Let $(X,\sd,x_\infty)$ be any tangent cone at infinity of $M$. 
    If $(\bb{R}^d,\sd_{eu})$ is a tangent space to $X$ at a point $x  \in X$, then $d \leq k-1$. 
    \item \label{item2blowdown}
    Assume that $\Ric_{n-2} \geq 0$ and
    \[
    \lim_{r \to + \infty} \fint_{B_r(p)} 0 \vee (r^{2-\alpha}\R) \wedge L \,\de\ssf{Vol}=L.
    \]
    Let $(X,\sd,x_\infty)$ be any tangent cone at infinity of $M$. 
    If $(\bb{R}^d,\sd_{eu})$ is a tangent space to $X$ at a point $x \in X$, then $d \leq n-2$.
    \end{enumerate}
    \begin{proof}
        \textbf{Proof of }\ref{item1blowdown}.
        Let $r_i \uparrow +\infty$ and consider the rescaled spaces $(M,g/r_i^2,p)$. We consider the manifolds $(M_i,g_i,p_i)$, where $M_i=M$, $g_i=g/r_i^2$, and $p_i=p$. We denote by $B_s^i(p_i)$ and $\ssf{Vol}_i$ respectively balls and volumes in $M_i$. By the scaling properties of curvature, these manifolds satisfy
        \begin{equation*}
            \fint_{B^i_1(p_i)} 0 \vee (r_i^{-\alpha}\R_{M_i,k}) \wedge L \, d\ssf{Vol}_i=
            \fint_{B_{r_i}(p)} 0 \vee (r_i^{2-\alpha}\R_{M,k}) \wedge L \, d\ssf{Vol},
        \end{equation*}
        so that by hypothesis
        \begin{equation} \label{bbb}
            \lim_{i \to + \infty}\fint_{B^i_1(p_i)} 0 \vee (r_i^{-\alpha}\R_{M_i,k})\wedge L \, d\ssf{Vol}_i=L.
        \end{equation}
        Let $(X,\sd,x_\infty)$ be a pGH limit of a (non-relabeled) subsequence of $(M_i,g_i,p_i)$.
        Let $x \in B_1(x_\infty)$ be a point where $\bb{R}^d$ is a tangent space of $X$. We claim that $d \leq k-1$. If this holds, then a rescaling argument implies the same dimensional bound on tangent spaces for every $x \in X$.
        
        Fix $\epsilon>0$, and let $s>0$ be small enough so that $(X,\sd/s,x)$ is $\epsilon$-close in pGH sense to $\bb{R}^d$. 
        Then, there exists points $p'_i \in M_i$ converging to $x$ such that, for $i$ large enough, $(M_i,g_i/s^2,p'_i)$ is $2\epsilon$-close in pGH sense to $\bb{R}^d$. 
        Modulo reducing $s$, we may assume that $B^i_s(p'_i) \subset B^i_1(p_i)$ for $i$ large enough.
        
        We claim that, for $i$ large enough, it holds
        \begin{equation} \label{Eqbl}
        \fint_{B^i_{s}(p'_i)}(s^2 \R_{M_i,k}) \wedge L \,d \ssf{Vol}_i \geq L/3.
        \end{equation}
        If this is true, by choosing $\epsilon>0$ small enough, we deduce that $d \leq k-1$ by Theorem \ref{CT1}.

        To prove \eqref{Eqbl}, it is enough to show that
            \begin{equation} \label{E38}
            \frac{\ssf{Vol}_i(B^i_s(p'_i) \cap \{\R_{M_i,k} \geq s^{-2}L/2\})}{\ssf{Vol}_i(B^i_s(p'_i))} \rightarrow 1, \quad \text{as }i \to + \infty.
            \end{equation}
            Assume by contradiction that \eqref{E38} fails. Then there exists $\delta>0$ such that, up to a non-relabeled subsequence: 
            \[
            \ssf{Vol}_i(B^i_s(p'_i) \cap \{\R_{M_i,k} \geq s^{-2}L/2\}) < (1-\delta)\ssf{Vol}_i(B^i_s(p'_i)).
            \]
            Hence, for $i$ large enough:
            \begin{align*}
            \int_{B^i_1(p_i)}  &0 \vee (r_i^{-\alpha}\R_{M_i,k}) \wedge L \, \de\ssf{Vol}_i \\
            & \leq L \ssf{Vol}_i(B^i_1(p_i) \setminus B^i_s(p'_i))+L(1-\delta)\ssf{Vol}_i(B^i_s(p'_i))+r_i^{-\alpha}s^{-2}L/2\ssf{Vol}_i(B^i_s(p'_i)).
            \end{align*}
            Using that $B^i_s(p'_i) \subset B^i_1(p_i)$ for $i$ large enough, we infer that
            \begin{align*}
            \int_{B^i_1(p_i)}  & 0 \vee (r_i^{-\alpha}\R_{M_i,k}) \wedge L \, \de\ssf{Vol}_i \\
            & \leq L \ssf{Vol}_i(B^i_1(p_i) )+(r_i^{-\alpha}s^{-2}L/2-\delta L)\ssf{Vol}_i(B_s(p'_i)).
            \end{align*}
            By Bishop-Gromov's inequality, we conclude that for a positive constant $c(s,n)>0$ it holds
            \begin{align*}
            \fint_{B^i_1(p_i)}  &0 \vee (r_i^{-\alpha}\R_{M_i,k}) \wedge L \, \de\ssf{Vol}_i \leq
            L+(r_i^{-\alpha}s^{-2}L/2-\delta L)c(s,n).
            \end{align*}
            Since $r_i \uparrow \infty$, this contradicts \eqref{bbb}, proving item \ref{item1blowdown}.

            \textbf{Proof of }\ref{item2blowdown}. Bounds as in \eqref{Eqbl} follow repeating the previous argument for the integral appearing in item \ref{item2blowdown}. To conclude, one then uses Theorem \ref{CT2} instead of Theorem \ref{CT1}. 
    \end{proof}
\end{thm}

\begin{remark} \label{R|extangent}
    In Theorem \ref{T|tangentcones}, if we replace the limit symbols with lim-sup symbols, we obtain that \emph{there exists} a tangent cone satisfying the corresponding dimensional condition. 

    Moreover, one can replace $r^{2-\alpha}$ in the statement with an arbitrary positive measurable function $f:\bb{R}_+ \to \bb{R}_+$ sucht that $f(r)/r^2 \to 0$ as $r \uparrow + \infty$.

    The proofs of this facts are the same as the one of Theorem \ref{T|tangentcones}.
\end{remark}

\begin{remark} \label{R3}
    Theorem \ref{T|tangentcones} fails for $\alpha=0$. We give a counterexample for $k=1$, and all the other cases follow by taking products with real factors. Indeed, consider the paraboloid $\Sigma:=\{(x,y,z) \in \bb{R}^3: z=x^2+y^2\}$. The sectional curvature of $\Sigma$ is $K(z) \sim 1/4z^2$ as $z \to + \infty$. For every $r>0$, the subset $\{z \leq r\} \subset \Sigma$ is a ball of radius $f(r) \geq r$. Hence, on the paraboloid, for a positive constant $c>0$, it holds $f(r)^2K(z) \geq c>0$ for all points in $\{z \leq r\}$. Hence, 
    \[
    \lim_{r \to + \infty} \fint_{B^{\Sigma}_{f(r)}(0)} (f(r)^2K) \wedge c \, \ssf{Vol} = c.
    \]
    At the same time, the tangent cone at infinity of the paraboloid is a half line.
\end{remark}

\begin{remark} \label{R4}
    Let $L>0$ and $\alpha \in (0,2)$ be fixed. Another immediate consequence of Theorem \ref{T|tangentcones} is that if a manifold $(M,g)$ with non-negative Ricci curvature has $\Ric_M \geq f $ for a function $f$ such that
    \[
    \lim_{r \to + \infty} \fint_{B_r(p)} 0 \vee (r^{2-\alpha}f) \wedge L \,\de\ssf{Vol}=L,
    \]
    then $M$ is compact. A similar result was obtained in \cite[Theorem 4.8]{CGT}, where also the case $\alpha=2$ was addressed, under a pointwise (rather than integral) bound on $\Ric_M$.
\end{remark}

We now turn our attention to the topological implications of Theorem \ref{T|tangentcones}. The next result corresponds to Theorem \ref{thmIntFinal} from the introduction.

\begin{thm} 
    Let $(M^n,g,p)$ be a non-compact Riemannian manifold with $n \geq 3$ and let $\alpha \in (0,2)$. Let $2 \leq k \leq n$ be a natural number.
    \begin{enumerate}
        \item \label{item1:bettiInf}
        If $\Ric_M \geq 0$ and
    \begin{equation} \label{EbettiInf1}
    \limsup_{r \to + \infty} \Big( r^{2-\alpha} \min_{B_r(p)} \R_k \Big) >0,
    \end{equation}
    then $\mathrm{b}_1(M) \leq k-2$. 
    \item \label{item2:bettiInf}
    If $\Ric_{n-2} \geq 0$ and
    \begin{equation} \label{EbettiInf2}
    \limsup_{r \to + \infty} \Big( r^{2-\alpha} \min_{B_r(p)} \R \Big) >0,
    \end{equation}
    then $\mathrm{b}_1(M) \leq n-3$.
    \end{enumerate}
    \begin{proof}
    \textbf{Proof of \ref{item1:bettiInf}}: Let $(\tilde{M},\tilde{g},\tilde{p})$ be a covering of $M$, and let $\pi:\tilde{M} \to M$ be the covering map. We claim that $\tilde{M}$ has a tangent cone at infinity of dimension $\leq k-1$.

    We first show that $\tilde{M}$
    satisfies condition \eqref{EbettiInf1}. Let $r>0$ and let $\tilde{x} \in \tilde{M}$ be a point with $\tilde{\sd}(\tilde{x},\tilde{p}) < r$. Then, $\pi(\tilde{x}) \in B_r(p)$, so that
    \[
    r^{2-\alpha} \min_{B_{r}(\tilde{p}) }\R_{\tilde{M},k} \geq r^{2-\alpha} \min_{B_{r}(p) }\R_{M,k}.
    \]
    Hence, also $\tilde{M}$ satisfies \eqref{EbettiInf1}.
    
    Let now $r_i \uparrow + \infty$ and let $\epsilon>0$ be such that 
    \begin{equation*}
    \lim_{i \to + \infty} \Big( r_i^{2-\alpha} \min_{B_{r_i}(\tilde{p})} \R_{\tilde{M},k} \Big) =\epsilon.
    \end{equation*}
    It follows that
    \[
    \lim_{i \to + \infty} \fint_{B_{r_i}(\tilde{p})} (r_i^{2-\alpha}\R_{\tilde{M},k}) \wedge \epsilon \, \widetilde{d\ssf{Vol}}=\epsilon.
    \]
    As a consequence, by item \ref{item1blowdown} of Theorem \ref{T|tangentcones} and Remark \ref{R|extangent}, $\tilde{M}$ has a tangent cone at infinity of dimension $\leq k-1$.
    The statement now follows by repeating the argument of \cite[Theorem 1.3]{zhu2024twodimensionvanishingsplittingpositive}, where instead of using \cite[Theorem 1.6]{zhu2024twodimensionvanishingsplittingpositive} one uses that $\tilde{M}$ has a tangent cone at infinity of dimension $\leq k-1$.

    \textbf{Proof of \ref{item2:bettiInf}}: The argument is analogous to the proof of item \ref{item1:bettiInf}, now using item \ref{item2blowdown} of Theorem \ref{T|tangentcones}.
    \end{proof}
    \end{thm}

Next, we prove Theorem \ref{thmInt4} from the Introduction. 

\begin{thm}
    Let $\epsilon,s,v \in (0,1)$, let $L,D \in (0,+\infty)$, and let $k,n \in \bb{N}$ with $k \leq n$. There exists $\delta>0$ such that for every manifold $(M^n,g,p)$ with
    $\Ric_{M} \geq -\delta$, $\mathrm{diam}(M) \leq D$, $\ssf{Vol}(B_1(p)) \geq v$, the following holds.
    \begin{enumerate}
    \item \label{item11} If
    $
     \fint_{B_1(p)} \R_{k} \wedge L \, \de \ssf{Vol} \geq \epsilon,
    $
    then $\mathrm{b}_1(M) \leq k-1$. 
        \item \label{item22} If
    $
     \fint_{B_1(p)} |\R|^s \, \de \ssf{Vol} \geq \epsilon,
    $
    then $\mathrm{b}_1(M) \leq n-2$. 
    \end{enumerate}
    \end{thm}
    
    \begin{proof}
We consider case \ref{item11} first. 
Assume by contradiction that the statement fails.
Hence, there exist $v,\, D, \epsilon_0>0$, a sequence  $\delta_j \downarrow 0$ and pointed Riemannian manifolds $(M^n_j,g_j,p_j)$ with $\Ric_{M_j} \geq -\delta_j$, $\mathrm{diam}(M_j) \leq D$, $\ssf{Vol}_j(B_1(p_j)) \geq v$, 
\[
\fint_{B_1(p_j)} \R_{M_j,k} \wedge L \, \de \ssf{Vol}_j \geq \epsilon_0,
\]
and $\mathrm{b}_1(M_j) \geq k$.
Up to a subsequence, $M_j$ GH-converge to a metric space $(X,\sd)$. 
Thanks to the non-collapsing assumption, the metric measure space $(X,\sd,\aH^n)$ is an $\mathrm{RCD}(0,n)$ space while, thanks to Proposition \ref{P|uppersemicontinuity betti}, it holds $\mathrm{b}_1(X) \geq k$.
By Proposition \ref{P|splitting universal cover}, the universal cover $(\tilde{X},\tilde{\sd})$ of $X$ splits $\bb{R}^k$ isometrically.

Let $\tilde{p} \in \tilde{X}$.
By Proposition \ref{P|convergence coverings}, the pointed metric space $(\tilde{X},\tilde{\sd},\tilde{p})$ is the  pGH-limit of a sequence $(\tilde{M}_j,\tilde{g}_j,\tilde{p}_j)$, where each $\tilde{M}_j$ covers $M_j$.
These covering spaces satisfy
\[
\int_{B_1(\tilde{p}_j)} {\R}_{\tilde{M}_j,k} \wedge L \, \de \widetilde{\ssf{Vol}}_j \geq 
\int_{B_1(p_j)} {\R}_{M_j,k} \wedge L \, \de \ssf{Vol}_j
\geq
v \epsilon_0 \geq c(n)v \epsilon_0 \widetilde{\ssf{Vol}}_j (B_1(\tilde{p}_j)),
\]
contradicting Theorem \ref{CT1}.

Assertion \ref{item22} follows in an analogous way, using Theorem \ref{CT2} instead of Theorem \ref{CT1}.
\end{proof}

\section{Thin metric spaces} \label{S1}

The goal of this section is to show that \emph{thin} metric spaces (see Definition \ref{Condition}, after \cite{TrianglesKapo}) are contained in a neighbourhood of controlled width of an isometrically embedded $1$-dimensional manifold (see Theorem \ref{thm(SNP)}). This will be a key step to prove Theorem \ref{thmIntMain}.

In \cite[Section 5] {TrianglesKapo}, it is shown that a \emph{non-compact} thin metric space is contained in a finite neighbourhood of a ray or a line. Comparing with \cite[Section 5] {TrianglesKapo}, Theorem \ref{thm(SNP)} treats the compact case as well, and provides explicit uniform bounds on the width of the aforementioned neighbourhood. These bounds are needed for the applications in the subsequent sections and, in particular, to prove Theorem \ref{thmIntMain}.
On the other hand, unlike \cite[Section 5] {TrianglesKapo}, we only consider metric spaces $(X,\sd)$ that are proper, separable, and geodesic.

If not otherwise specified, all curves are assumed to be of finite length and parametrized by arc-length. The length functional is denoted by $\length(\cdot)$.

\begin{definition}[Ray, segment, line]
    Let $(X,\sd)$ be a metric space. A function $r:[0,+\infty) \to X$ is called  \emph{a ray} if
    $
    \sd(r(t),r(s))=|t-s|
    $
    for every $t,s \in [0,+\infty)$.
    Similarly $r:[a,b] \to X$ is called \emph{a segment} if
    $
    \sd(r(t),r(s))=|t-s|
    $
    for every $t,s \in [a,b]$ and  $r:\bb{R} \to X$ is called  \emph{a line} if
    $
    \sd(r(t),r(s))=|t-s|
    $
    for every $t,s \in \bb{R}$.
\end{definition}

When considering a ray (resp.\;a segment or a line) $r$ in $X$, by a slight abuse in order to keep notation short, we denote by $r$ both the function $r:[0,+\infty) \to X$ and its image. Moreover,  $[r(a),r(b)]$ denotes the set $r([a,b]) \subset X$,  for every $0<a<b$.

Given a closed set $K \subset X$ and a point $x\in X$, we call \emph{the projection} of $x$ into $K$ the subset $\pi_{K}(x)\subset K$ defined by
$$
\pi_{K}(x):=\{y\in K\colon \sd(x, K)=\sd(x,y)\}.
$$
Given a subset $A \subset X$, we call \emph{the projection} of $A$ into $K$ the subset $\pi_{K}(A)\subset K$ defined by
\[
\pi_K(A):=\bigcup_{x \in A} \pi_K(x).
\]
Moreover, for every $y \in K$, we denote by $\pi^{-1}_K(y) \subset X$ the subset
\begin{equation}\label{D3}
\pi^{-1}_K(y):=\{x\in X\colon \sd(x,K)=\sd(x,y)\}.
\end{equation}

\begin{definition}[Thin metric spaces]\label{Condition}
Let $R,D>0$ with $R \geq 20D$. Let $(X,\sd)$ be a proper, geodesic and separable metric space. We say that $(X,\sd)$ is \emph{$(R,D)$-thin} if for every segment $r$ of length greater than $2R$ in $X$, and for every $t \in (R,\length(r)-R)$ and $x \in \pi_r^{-1}(r(t))$, it holds $\sd(x,r) < D$.
\end{definition}

In \cite{TrianglesKapo}, a metric space is defined to be $R$-thin if it is $(R,R)$-thin according to Definition \ref{Condition}. We consider the extra parameter $D>0$ and we restrict ourselves to the case $R \geq 20D$ as this will be sufficient for the subsequent applications.

\begin{remark} \label{R1}
Let $(X,\sd)$ be $(R,D)$-thin (see Definition \ref{Condition}) and
    let $r$ be a ray in $X$.
    If a point $x \in X$ satisfies $\pi_r(x) \cap (r(R),r(\infty)) \neq \emptyset$, then $\sd(x,r) \leq D$.
    This fact follows by applying Definition \ref{Condition} to the segments $[r(0),r(T)]$ for $T>2R$.
\end{remark}

 The next lemma is needed to prove Lemma \ref{nuovocontinuitylemma}, which is the key technical tool to study $(R,D)$-thin metric spaces.
 
\begin{lemma} \label{L20}
Let $(X,\sd)$ be a geodesic metric space and let $C_1 \subset C_2 \subset X$ be closed sets. The set
\[
    K:=\{ x \in X:  \pi_{C_2}(x) \cap  C_1 \neq \emptyset\}
    \]
    is closed.
    
In particular, if $\alpha:[0,b] \to X$ is a curve, then
the set
    \[
    I:=\{ s \in [0,b]: \pi_{C_2}(\alpha(s)) \cap C_1 \neq \emptyset\}
    \]
    is closed as well.
    \begin{proof}
        Let $\{x_i\}_{i \in \bb{N}} \in K$ be a sequence converging to $x_\infty \in X$. We claim that $x_\infty \in K$. 
        For every $i \in \bb{N}$, let $\gamma_i$ be a segment from $x_i$ to a point $\gamma_i(\length(\gamma_i)) \in C_1$ such that $\sd(x_i,C_2)=\length(\gamma_i)$.
 If $\length(\gamma_i) \to 0$ as $i \to + \infty$, then $x_\infty \in C_1$, so that $x_\infty \in K$.

        If instead $\length(\gamma_i)\not \to 0$, modulo passing to a subsequence, the segments $\gamma_i$ converge to a finite segment $\gamma$ from $x_\infty$ to a limit point of the sequence $\{\gamma_i(\length(\gamma_i))\}_{i \in \bb{N}}$. Since $C_1$ is closed, all limit points of  $\{\gamma_i(\length(\gamma_i))\}_{i \in \bb{N}}$ lie in $C_1$, so that also $\gamma(\length(\gamma)) \in C_1$. By lower semicontinuity of the length:
        \begin{equation*}
        \sd(x_\infty,C_2) \leq \length(\gamma) \leq \liminf_{i \in \bb{N}} \length(\gamma_i) =
        \liminf_{i \in \bb{N}} \sd(x_i,C_2)=
        \sd(x_\infty,C_2).
        \end{equation*}
        In particular,
        $\sd(x_\infty,C_2) =\length(\gamma)$, so that $x_\infty \in K$. Hence, $K$ is closed.
Being the preimage of $K$ via the continuous function $\alpha$, also $I$ is closed.
    \end{proof}
\end{lemma}

The next lemma contains the key technical properties of $(R,D)$-thin metric spaces (see Definition \ref{Condition}).

\begin{lemma} \label{nuovocontinuitylemma}
    Let $(X,\sd)$ be an $(R,D)$-thin metric space (see Definition \ref{Condition}). Let $\alpha:[0,\length(\alpha)] \to X$ be a unit speed curve.
    \begin{enumerate}
        \item \label{item1Lcont}
        Let $l$ be a segment with $\length(l)>2R$. Let $t \in (R,\length(l)-R)$, and let $u \in [0,R] \cup [\length(l)-R,\length(l)]$. If $l(t) \in \pi_l(\alpha(0))$ and $l(u) \in \pi_l(\alpha(\length(\alpha)))$, then the set $\pi_l(\alpha)$ contains a $3D$-net of either $[l(R),l(t)]$ or $[l(t),l(\length(l)-R)]$.
        \item \label{item2Lcont}
        Let $l$ be a segment with $\length(l)>2R$. Let $t,u \in (R,\length(l)-R)$ with $t<u$. If $l(t) \in \pi_l(\alpha(0))$ and $l(u) \in \pi_l(\alpha(\length(\alpha)))$, then $\pi_l(\alpha)$ contains a $3D$-net of either $[l(t),l(u)]$ or $[l(R),l(t)] \cup [l(u),l(\length(l)-R)]$.
        \item \label{item3Lcont}
        Let $r$ be a ray, let $t>R$ and let $u \in [0,R]$. If $r(t) \in \pi_r(\alpha(0))$ and $r(u) \in \pi_r(\alpha(\length(\alpha)))$, then $\pi_r(\alpha)$ contains a $3D$-net of $[r(R),r(t)]$.
        \end{enumerate}
        \begin{proof}
            We consider item \ref{item1Lcont} first.
            Assume by contradiction that the claim fails. Hence, there are $x_1 \in [R,t]$ and $x_2 \in [t,\length(l)-R]$ such that 
            \begin{equation} \label{EqLcont}
            \pi_l(\alpha) \cap B_{3D}(l(x_1))=\pi_l(\alpha) \cap B_{3D}(l(x_2))=\emptyset.
            \end{equation}
Consider
\begin{align*}
\tilde{s}:=\max \Big\{s \in [0,&\length(\alpha)]: \pi_l(\alpha(s)) \cap [l(x_1),l(x_2)] \neq \emptyset \Big\}.
\end{align*}
In order to show that $\tilde{s}$ is well-defined, it is sufficient to observe that the set appearing in its definition is:
\begin{itemize}
\item \emph{closed}, thanks to Lemma \ref{L20}; 
\item  \emph{non-empty}, since by assumption $l(t) \in \pi_l(\alpha)$.
\end{itemize}
We claim that $\tilde{s} < \length(\alpha)$. If this is not the case, then, combining with \eqref{EqLcont}, 
\[
\pi_l\big(\alpha(\length(\alpha))\big) \cap [l(x_1+3D),l(x_2-3D)] \neq \emptyset.
\]
Since $X$ is $(R,D)$-thin, $\alpha(\length(\alpha))$ is $D$-close to $[l(x_1+3D),l(x_2-3D)]$, so that 
\[
\pi_l\big(\alpha(\length(\alpha)) \big) \cap l\big([0,R] \cup [\length(l)-R,\length(l)] \big)=\emptyset,
\]
a contradiction. This shows that $\tilde{s} < \length(\alpha)$.

Hence, for every $\epsilon>0$, there is $p_\epsilon \in \pi_l\big(\alpha(\tilde{s}+\epsilon) \big)$.  Let $p \in \pi_l\big(\alpha(\tilde{s}) \big) \cap [l(x_1),l(x_2)]$. Since $X$ is $(R,D)$-thin, it holds $\sd(p,\alpha(\tilde{s}))<D$. It follows that $\sd(p_\epsilon,\alpha(\tilde{s}+\epsilon))<D+\epsilon$. Combining these facts, it holds
\[
\sd(p,p_\epsilon) \leq \sd(p,\alpha(\tilde{s}))+\sd(\alpha(\tilde{s}),\alpha(\tilde{s}+\epsilon))+\sd(\alpha(\tilde{s}),p_\epsilon) \leq 2D+2\epsilon.
\]
This contradicts the fact that $p_\epsilon \in [l(0),l(x_1-3D)] \cup [l(x_2+3D),l(\length(l))]$ by \eqref{EqLcont}. 

The proof of item \ref{item2Lcont} is very similar to the one of item \ref{item1Lcont}, and for this reason it is only sketched.
Assume by contradiction that the claim fails. Hence, there are $x_1 \in [t,u]$ and $x_2 \in [R,t] \cup [u,\length(l)-R]$ such that, as in \eqref{EqLcont},
            \begin{equation*} 
            \pi_l(\alpha) \cap B_{3D}(l(x_1))=\pi_l(\alpha) \cap B_{3D}(l(x_2))=\emptyset.
            \end{equation*}
One can now assume without loss of generality that $x_2 \geq x_1$ (the other case being analogous), and the proof can be carried out in the same way as for item \ref{item1Lcont}.

We now consider item \ref{item3Lcont}. Since the image of $\alpha$ is compact, $\pi_l(\alpha)$ is a bounded subset of $l$. Hence, $\pi_l(\alpha) \subset [r(0),r(T)]$ for some $T>0$. We consider now the segment $l:=[r(0),r(T+2R)]$. It is easy to check that $\pi_l(\alpha)=\pi_r(\alpha)$. By item \ref{item1Lcont}, $\pi_l(\alpha)$ contains a $3D$-net of either $[l(R),l(t)]$ or $[l(t),l(\length(l)-R)]$. This last case cannot happen as 
$[l(t),l(\length(l)-R)]=[r(t),r(T+R)]$ and $\pi_l(\alpha) \subset [r(0),r(T)]$. Hence, $\pi_l(\alpha)$ (which again coincides with $\pi_r(\alpha)$) contains a $3D$-net of $[l(R),l(t)]=[r(R),r(t)]$.
        \end{proof}
\end{lemma}

The next few lemmas are needed to prove Proposition \ref{P10}, which gives the macroscopic description of $(R,D)$-thin non-compact spaces. 

\begin{lemma} \label{L|alternative connectedness}
    Let $(X,\sd)$ be an $(R,D)$-thin metric space (see Definition \ref{Condition}).
    Let $r$ be a ray in $X$ and let $y \in X$ with $\sd(r,y)>D$. Let $t>0$ and let $\gamma$ be a curve connecting $y$ and $r(t)$. Then, for every $s \in [0,t]$, it holds $\sd(r(s),\gamma) \leq 2R$. 
    \begin{proof} 
We parametrize $\gamma$ by arc length in such a way that $\gamma(0)=r(t)$ and $\gamma(\length(\gamma))=y$. Since $X$ is $(R,D)$-thin and $\sd(y,r)>D$, it follows that $\pi_r(\gamma(\length(\gamma))) \in [r(0),r(R)]$. By Lemma \ref{nuovocontinuitylemma}, $\pi_r(\gamma)$ contains a $3D$-net of $[r(R),r(t)]$. Using again that $X$ is $(R,D)$-thin, it follows that $[r(R),r(t)]$ is in a $7D$-neighbourhood of $\gamma$. The statement then follows.
    \end{proof}
\end{lemma}

\begin{definition}[Divergent rays]
    We say that two rays $r_1,r_2$ in a metric space $(X,\sd)$ are \emph{divergent} if 
   \[
   \limsup_{t \to + \infty} \sd(r_1(t),r_2(t))=+\infty.
   \]
\end{definition}

\begin{lemma} \label{L5}
  Let $(X,\sd)$ be an $(R,D)$-thin metric space (see Definition \ref{Condition}).
     If there are two divergent rays in $X$, then $X$ contains a line.
    \begin{proof}
       Let $r_1,r_2$ be divergent rays. 
        Consider a sequence $t_j \uparrow + \infty$ such that $\sd(r_1(t_j),r_2(t_j)) \to + \infty$ and let $l_j$ be segments from $r_1(t_j)$ to $r_2(t_j)$. 
        By Lemma \ref{L|alternative connectedness}, each one of these segments satisfies
        \begin{equation} \label{E9}
        \sd(l_j,r_1(0)) \leq c(R,D).
        \end{equation}
        Since $\sd(r_1(0),r_1(t_j)) \to + \infty$ and $\sd(r_1(0),r_2(t_j)) \to + \infty$, condition \eqref{E9} implies that $l_j$ converges (modulo passing to a subsequence) to a line. This concludes the proof.
    \end{proof}
\end{lemma}

\begin{lemma} \label{L4}
Let $(X,\sd)$ be an $(R,D)$-thin metric space (see Definition \ref{Condition}).
     Let $r$ be a ray in $X$. Let $y \in X$ be a point such that $\sd(y,r)>D$. Let $r_y$ be a ray obtained as a limit of the segments connecting $y$ and $r(t_j)$, for some sequence  $t_j \to + \infty$. Then, $r$ is contained in the  $2R$-neighbourhood of $r_y$.
     \begin{proof}
         It follows immediately from Lemma \ref{L|alternative connectedness}.
     \end{proof}
\end{lemma}

The next proposition shows that an $(R,D)$-thin non-compact metric space is contained in a neighbourhood of controlled width of a ray or a line. We recall that a similar result is proved in \cite[Section 5]{TrianglesKapo}. The main difference is that Proposition \ref{P10} below also gives an explicit bound on the width of the aforementioned neighbourhood.

\begin{proposition} \label{P10} 
    Let $(X,\sd)$ be an $(R,D)$-thin metric space (see Definition \ref{Condition}).
     Assume that $X$ is non-compact and it does not contain a line.
     There exists a ray $r$ in $X$ whose $10R$-neighbourhood contains $X$.
        \end{proposition}
        \begin{proof}
            We first show that, given any ray $r$ in $X$, there exists $D'(X,r)>0$ such that $X$ is contained in the $D'$-neighbourhood of $r$.
            
            If the claim were false, we would find points $p_j$ at arbitrarily large distance from $r$. 
            For each point $p_j$, let $r_j$ be a ray  obtained as a limit of the segments from $p_j$ to $r(t_i)$ for some sequence $t_i \to + \infty$. 
           By Lemma \ref{L|alternative connectedness}, each ray $r_j$ has a point $q_j$ whose distance from $r(0)$ is less than $2R$. 
            Since the distance between $p_j=r_j(0)$ and $q_j$ goes to infinity as $j$ increases, and the sequence $(q_j)$ is contained in a compact subset of $X$, then, up to a subsequence,  the rays $r_j$ converge to a line  $l$ contained in $X$, a contradiction.

           We now prove the statement of the lemma.
           If $X$ is contained in the $10R$-neighbourhood of $r$, there is nothing to prove. 
           Otherwise, we claim that there exists an element $y \in X$ at maximal distance from $r$. 
            Let $y_i \in X$ be a sequence maximizing the distance from $r$. We can assume that $9R \leq \sd(y_i,r) \leq D'(X)$. Since $X$ is $(R,D)$-thin, $\pi_r(y_i) \subset [r(0),r(R)]$. It follows that the sequence $\{y_i\}_{i \in \bb{N}}$ is precompact, so that there exists a limit point $y \in X$ at maximal distance from $r$.

           Let $r_y$ be the ray which is the limit of the segments joining $y$ and $r(t_i)$ for a sequence $t_i \uparrow + \infty$. By Lemma \ref{L4}, for every $t \geq 0$, it holds $\sd(r(t),r_y) \leq 2R$. Since $\sd(y,r) \geq 10R$, it follows that $\pi_{r_y}(r) \subset [r_y(5R),r_y(+\infty))$. 

           Assume now by contradiction that there exists $y' \in X$ such that $\sd(y',r_y) \geq 10R$. Let $\gamma$ be a segment realizing the distance between $y'$ and $r$. Let $r(t_\gamma)$ be the endpoint of $\gamma$. Let $r_y(t_\gamma^y) \in \pi_{r_y}\big(r(t_\gamma) \big)$. By our previous remarks, $t_\gamma^y \geq 5R$. By Lemma \ref{nuovocontinuitylemma}, $\pi_{r_y}(\gamma)$ contains a $3D$-neighbourhood of $[r_y(R),r_y(5R)]$. Hence, $\gamma$ contains a point $p_1$ which is $7D$-close to $r_y(R)$. Recall that we also have $\sd(y',r_y) \geq 10R$. Hence, 
           \begin{align*}
           \sd(y',r)&=\length(\gamma) = \sd(y',p_1)+\sd(p_1,r) 
           \geq \sd(y',r_y)-7D+\sd(y,r)-R-7D \\ 
           & \geq \sd(y,r)+9R-14D>\sd(y,r)
           \end{align*}
           Hence, $y$ was not an element maximizing the distance from $r$, a contradiction.
        \end{proof}

We now turn our attention to the study of compact $(R,D)$-thin metric spaces. The key result in this sense is Proposition \ref{IsometriaCerchio}. We first need a few lemmas.

\begin{lemma} \label{L15}
    Let $(X,\sd)$ be an $(R,D)$-thin metric space.
     Let $l$ be a segment in $X$ and assume that $\mathrm{length}(l) \geq 10R$. Let $\bar{l}$ be a segment starting from $l(0)$ with $\pi_l(\bar{l}(\length(\bar{l}))) \subset [l(0),l(R)] \cup [l(\length(l)-R),l(\length(l)]$. 
     If $\pi_l\big(\bar{l}\big) \cap [l(3R),l(\length(l)-3R)] \neq \emptyset$,
      then $\pi_l(\bar{l})$ contains a $3D$-net of $[l(R),l(\length(l)-R)]$.
     \begin{proof}
         Let $t_0 \in [0,\length(\bar{l})]$ be such that $\pi_l\big(\bar{l}(t_0)\big) \ni l(t)$ for $t \in [3R,\length(l)-3R]$.
Let $\bar{l}_1:=\bar{l}_{|[0,t_0]}$ and $\bar{l}_2:=\bar{l}_{|[t_0,\length(\bar{l})]}$. 
Hence, both $\bar{l}_1$ and $\bar{l}_2$ have an endpoint whose projection on $l$ lies in $[l(0),l(R)] \cup [l(\length(l)-R),l(\length(l)]$.

By item \ref{item1Lcont} of Lemma \ref{nuovocontinuitylemma}, $\pi_l(\bar{l}_1)$ contains a $3D$-net of either $[l(R),l(t)]$ or $[l(t),l(\length(l)-R)]$, and the same holds for $\pi_l(\bar{l}_2)$.

Assume by contradiction that both $\pi_l(\bar{l}_1)$ and $\pi_l(\bar{l}_2)$ contain a $3D$-net of $[l(R),l(t)]$. In this case, there exist $t_1 \leq t_0 \leq t_2$ such that $\bar{l}(t_1)$ and $\bar{l}(t_2)$ are $7D$-close to $l(R)$. Hence,
\begin{equation} \label{eqf}
    \sd(\bar{l}(t_1),\bar{l}(t_2)) \leq 14D.
\end{equation}
Since $\bar{l}(t_0)$ is $D$-close to a point $l(t)$ in $[l(3R),l(\length(l)-3R)]$, it follows
\[
t_0-t_1=\sd(\bar{l}(t_0),\bar{l}(t_1)) \geq  \sd(l(R),l(t))-8D \geq 2R-8D,
\]
so that
$\sd(\bar{l}(t_1),\bar{l}(t_2)) \geq 2R-5D.$ This violates \eqref{eqf}. 

By the same argument, it is not possible that both $\pi_l(\bar{l}_1)$ and $\pi_l(\bar{l}_2)$ contain a $3D$-net of $[l(t),l(\length(l)-R)]$, concluding the proof.
     \end{proof}
\end{lemma}

\begin{lemma} \label{L9}
    Let $(X,\sd)$ be an $(R,D)$-thin metric space.
     Let $l$ be a segment in $X$ and assume that $\mathrm{length}(l) \geq 10R$. Let $p \in X$ be a point with $\sd(p,l) \geq 10R$ and let $\bar{l}$ be a segment connecting $p$ to $l(0)$.
    If $\pi_l\big(\bar{l}\big) \cap [l(3R),l(\length(l)-3R)] \neq \emptyset$, then  $\length(\bar{l}) > \length(l)$.
    \begin{proof}
Since $\sd(p,l) \geq 10R$ and $X$ is $(R,D)$-thin, it follows that $\pi_l(p) \subset [l(0),l(R)] \cup [l(\length(l)-R),l(\length(l)]$.
By Lemma \ref{L15}, there is a point $\bar{l}(s_0)$ which is $7D$-close to $l(\length(l)-R)$. Using triangle inequality, this implies
\[
\length(\bar{l})=\sd(p,\bar{l}(s_0))+\sd(\bar{l}(s_0),l(0)) \geq 
\sd(p,l)-7D+\sd(l(0),l(\length(l)-R))-7D.
\]
Using that $\sd(p,l) \geq 10R$ and that $l$ is a segment, it follows
$
\length(\bar{l}) \geq 9R-14D+\length(l)>\length(l).
$
    \end{proof}
\end{lemma}

\begin{lemma} \label{L16}
Let $(X,\sd)$ be a compact $(R,D)$-thin metric space.
     Let $l$ be a segment of maximal length in $X$ and assume that $\mathrm{length}(l) \geq 10R$. Let $p \in X$ be a point with $\sd(p,l) \geq 10R$ and let $\gamma_1,\gamma_2$ be segments connecting $p$ respectively to $l(\length(l))$ and $l(0)$.
    Let $b_1,b_2 \in \{\gamma_1,\gamma_2,l\}$ with $b_1 \neq b_2$. Then $\pi_{b_1}(b_2) \subset [b_1(0),b_1(3R)] \cup [b_1(\length(b_1)-3R),b_1(\length(b_1)]$.
    \begin{proof}
        If $b_1=l$, the statement follows by Lemma \ref{L9}.
        
        We now consider the case $b_1 \in \{\gamma_1,\gamma_2\}$ and $b_2=l$. If the statement fails, $\pi_{b_1}\big(l\big) \cap (b_1(3R),b_1(\length(b_1)-3R)) \neq \emptyset$. Consider the orientation of $l$ such that $l(0) \in b_1$. Then, $\pi_{b_1}(l(\length(l))) \subset [b_1(0),b_1(R)] \cup [b_1(\length(b_1)-R),b_1(\length(b_1)]$ since otherwise, using that $X$ is $(R,D)$-thin, it follows that $l$ is shorter than $b_1$.
        Hence, by Lemma \ref{L15}, there exists $s>0$ such that $l(s)$ is $R+7D$-close to $p$. This contradicts that $\sd(p,l) \geq 10R$.
        
        Finally, we consider the case $b_1=\gamma_1$ and $b_2=\gamma_2$ (the case $b_2=\gamma_1$ and $b_1=\gamma_2$ is identical). 
        
        We consider orientations such that $\gamma_1(\length(\gamma_1))=\gamma_2(0)$, $\gamma_2(\length(\gamma_2))=l(0)$ and $l(\length(l))=\gamma_1(0)$.
        
        Assume by contradiction that $\gamma_1(t) \in \pi_{\gamma_1}(\gamma_2)$, with $t \in (3R,\length(\gamma_1)-3R)$. By Lemma \ref{L15}, there exists $s>0$ such that $\gamma_2(s)$ is $R+7D$-close to $\gamma_1(0)=l(\length(l))$. It follows that
        \[
        \length(\gamma_2)
        =\sd(p,
        \gamma_2(s))
        +\sd(\gamma_2(s),l(0)) \geq \sd(p,l)-R-7D + \length(l)-R-7D >\length(l).
        \]
        This is a contradiction.
    \end{proof}
    \end{lemma}

    \begin{corollary} \label{C5}
        Let $(X,d)$ and $\gamma_1, \, \gamma_2, \,l$ be as in the previous lemma. Let $\{b_1,b_2,b_3\}=\{\gamma_1,\gamma_2,l\}$. If $b_1(0) \notin b_2$, the follows happen.
        \begin{enumerate}
        \item
        $
\sd(b_1(0),b_2(\length(b_2)/2))
        \geq 
        \min \{\length(b_1),\length(b_3)\}+\length(b_2)/2-10R.
        $
        \item $
\sd(b_1(\length(b_1)/2)),b_2(\length(b_2)/2))
        \geq 
        \length(b_1)/2+\length(b_2)/2-10R.
        $
        \end{enumerate}
        \begin{proof}
        We only prove the first point, the proof of the second one being analogous.
            Let $\gamma$ be a segment connecting $b_1(0)$ and $b_2(\length(b_2)/2)$. If $\gamma$ has a point which is $5R$-close to either $b_2(3R)$ or $b_2(\length(b_2)-3R)$, the statement follows by triangle inequality.

            By the previous lemma, $\pi_{b_2}\big(b_1(0)\big) \cap [b_2(3R),b_2(\length(b_2)-3R] \neq \emptyset$. If $\pi_{b_2}\big(b_1(0)\big) $ intersects $[b_2(R),b_2(3R)]$ or $[b_2(\length(b_2)-3R),b_2(\length(b_2)-R)]$, then $\gamma$ has a point which is $5R$-close to either $b_2(3R)$ or $b_2(\length(b_2)-3R)$ (because $X$ is $(R,D)$-thin), otherwise the same conclusion follows by Lemma \ref{nuovocontinuitylemma}.
        \end{proof}
    \end{corollary}

\begin{proposition} \label{IsometriaCerchio}
    Let $(X,\sd)$ be a compact $(R,D)$-thin metric space (see Definition \ref{Condition}).
     Let $l$ be a segment of maximal length in $X$ and assume that $\length(l) \geq 200R$. If $X$ is not contained in a $200R$-neighbourhood of $l$, then there exists $L>0$ and a distance preserving map $\phi:S^1_L \to X$, such that $X$ is contained in the $D$-neighbourhood of $\phi(S^1_L)$. In particular, $L \geq 50R$.
     \begin{proof}
         Let $p \in X$ with $\sd(p,X)>200R$. Let $\gamma_1$ and $\gamma_2$ be segments from $p$ to $l(\length(l))$ and  $l(0)$ respectively. Note that both these segments have length greater than $200R$.

         On $\gamma_1$, we consider the subsegment
         \[
         I^{\gamma_1}:=[\gamma_1(\length(\gamma_1)/2-5R)),\gamma_1(\length(\gamma_1)/2+5R))].
         \]
         We also consider subsegments $I^{\gamma_2}$ and $I^{l}$ defined in the same way. Observe that these subsegments are independent of the chosen  parametrization by arc-length of $\gamma_1$, $\gamma_2$ and $l$.
         
         To construct the map $\phi$ of the statement, we will consider the shortest loop passing near each one of the previously defined segments. To this aim, let
         \begin{align*}
         A:=\Big\{ \phi: &[0,\length(\gamma_1+\gamma_2+l)] \to X\mid \phi \text{ is } 1\text{-Lipschitz, } \phi(0)=\phi(\length(\gamma_1+\gamma_2+l)),
         \\
         & \quad
         \exists \, 
           t^{\gamma_1} , \,t^{\gamma_2} , 
           \, t^l \in [0, \length(\gamma_1+\gamma_2+l)] \text{ such that }
          \\
          & \quad \pi_{b}\big(\phi(t^b)\big) \cap I^b \neq \emptyset  \text{ for every } b \in \{\gamma_1,\gamma_2,l\} \Big\}.
         \end{align*}
         Notice that $A \neq \emptyset$ since $\gamma_1+\gamma_2+l \in A$. We claim that if $\{\phi_i\}_{i \in \bb{N}} \subset A$ and $\phi_i \to \phi_\infty$ uniformly, then $\phi_\infty \in A$. First of all $\phi_\infty$ is $1$-Lipschitz and satisfies $\phi_\infty(0)=\phi_\infty(\gamma_1+\gamma_2+l)$ trivially. Let $b \in \{\gamma_1,\gamma_2,l\}$. Consider for every $i \in \bb{N}$ the corresponding $t_{i}^{b}$. Modulo passing to a subsequence, there exists $t_{\infty}^{b} \in [0,\length(\gamma_1+\gamma_2+l)]$ such that $t_{i}^{b} \to t_{\infty}^{b}$. By Lemma \ref{L20}, $\pi_{b}\big( \phi_\infty(t_{\infty}^b) \big) \cap I^{b} \neq \emptyset$. 
         It follows that $\phi_\infty \in A$ as claimed.

         Consider now a sequence $\phi_i \in A$ minimizing the length functional. Each image of $\phi_i$ satisfies $\pi_{\gamma_1}\big( \mathrm{Im}(\phi_i)\big) \cap I_1^{\gamma_1} \neq \emptyset$. Since $X$ is $(R,D)$-thin, the image of each $\phi_i$ has a point which is $D$-close to $I_1^{\gamma_1}$. Hence, all the images of the maps $\phi_i$ are contained in a bounded set of $X$. Since the $\phi_i$ are all $1$-Lipschitz, by Ascoli-Arzelà theorem, there exists (modulo passing to a subsequence) $\phi_\infty$ such that $\phi_i \to \phi_\infty$ uniformly. By the previous part of the proof, $\phi_\infty \in A$. Since the length functional is lower semicontinuous w.r.t. uniform convergence, $\phi_\infty$ is an element of minimal length in $A$. 
         
         Let $T:=\length(\phi_\infty)$ and consider the reparametrizaion $\phi_\infty:[0,T] \to X$ by arc-length, and the induced map (denoted again by $\phi_\infty$) $\phi_\infty:S^1_L \to X$, where $L:=T/2\pi$. By Corollary \ref{C5}, it follows that $T \geq 60R$, and that the distance in $S^1_L$ between any two points of $\{t^{\gamma_1},t^{\gamma_2},t^l\}$ is at least $20R$ (in particular they are all distinct).

         We claim that $\phi_\infty:S^1_L \to X$ is distance preserving and that $X$ is contained in the $D$-neighbourhood of $\phi_\infty(S^1_L)$. We divide the proof of the claim in steps.
         
        \textbf{Step 1.} We claim that, for $b \in \{\gamma_1,\gamma_2,l\}$, the projection $\pi_b(\phi_\infty(S^1_L))$ contains a $3D$-net of $[b(3R),b(\length(b)-3R)]$.

        Let $b_1,b_2 \in \{\gamma_1,\gamma_2,l\}$ with $b_1 \neq b_2$. 
        Let $t^{b_1},t^{b_2} \in [0,T]$ be the numbers given in the definition of $A$ relative to $\phi_\infty$. Modulo reparametrizing $\phi_\infty$ and $\gamma_1+\gamma_2+l$, we can assume that $b_1(\length(b_1))=b_2(0)$ and $t^{b_1} \leq t^{b_2}$ with $t^{b_3} \notin [t^{b_1},t^{b_2}]$, i.e. $[t^{b_1},t^{b_2}]$ corresponds to the arc of $S^1_L$ that connects $t^{b_1}$ and $t^{b_2}$ without crossing $t^{b_3}$.
        
        Let $b_1(\bar{t}^{b_1}) \in I^{b_1} \cap \pi_{b_1}\big(\phi_\infty(t^{b_1})\big)$.
        Similarly, let $b_2(\bar{t}^{b_2}) \in I^{b_2} \cap \pi_{b_2}\big(\phi_\infty(t^{b_2})\big)$.
        We show that $\pi_{b_1}(\phi_\infty(S^1_L))$ contains a $3D$-net of $[b_1(\bar{t}^{b_1}),b_1(\length(b_1)-3R)]$, while $\pi_{b_2}(\phi_\infty(S^1_L))$ contains a $3D$-net of $[b_2(3R),b_2(\bar{t}^{b_2})]$. The claim then follows by repeating the same argument for all pairs of segments in $\{\gamma_1,\gamma_2,l\}$.
        
        By definition of $\phi_\infty$, the restriction ${\phi_\infty}_{|[t^{b_1},t^{b_2}]}$ is a segment. By Lemmas \ref{L16} and \ref{nuovocontinuitylemma}, $\pi_{b_1}({\phi_\infty}([t^{b_1},t^{b_2}]))$ contains a $3D$-net of either $[b_1(3R),b_1(\bar{t}^{b_1})]$ or $[b_1(\bar{t}^{b_1}), b_1(\length(b_1)-3R)]$.
        If the former case happens, ${\phi_\infty}([t^{b_1},t^{b_2}])$ has a point which is $7D$-close to $b_1(3R)$, so that it also contains a point $q$ which is $5R$-close to $b_1(0)$. 
        By Corollary \ref{C5}, $b_1(0)$ is at least $\min\{\length(b_1), \length(b_3)\}+\length(b_2)/2-20R$ distant from $\phi_\infty(t^{b_2})$. Hence,
        \begin{align*}
        \length({\phi_\infty}_{|[t^{b_1},t^{b_2}]}) & = \sd(\phi_\infty(t^{b_1}),q)+\sd(q,b_2(t^{b_2})) \\
        &\geq \length(b_1)/2+\min\{\length(b_1), \length(b_3)\}+\length(b_2)/2-50R
        \\
        & \geq \length(b_1)/2+\length(b_2)/2+100R.
        \end{align*}
        This is a contradiction, because we would obtain a shorter curve between $\phi_\infty(t^{b_1})$ and $\phi_\infty(t^{b_2})$ by first joining them to their projections $b_1(\bar{t}^{b_1})$ and $b_1(\bar{t}^{b_2})$, and then joining the projections along $b_1$ and $b_2$ themselves.
        
        Hence, $\pi_{b_1}({\phi_\infty}([t^{b_1},t^{b_2}]))$ contains a $3D$-net of $[b_1(\bar{t}^{b_1}), b_1(\length(b_1)-3R)]$. The same argument shows that $\pi_{b_2}({\phi_\infty}([t^{b_1},t^{b_2}]))$ contains a $3D$-net of $[b_2(3R),b_2(\bar{t}^{b_2})]$. This concludes the proof of Step 1.

        \textbf{Step 2.} For every $b_1,b_2 \in \{\gamma_1,\gamma_2,l\}$, with $b_1 \neq b_2$, assuming as before that $t^{b_1} \leq t^{b_2}$ in $[0,T]$ with $t^{b_3} \notin [t^{b_1},t^{b_2}]$, it holds $\pi_{\phi_\infty([t^{b_1},t^{b_2}])}(\phi_\infty(t^{b_2}+3R)) =\{\phi_\infty(t^{b_2})\}$.

        Modulo reparametrizing, we assume $b_1(\length(b_1))=b_2(0)$ and $t^{b_1} \leq t^{b_2} \leq t^{b_3}$ in $[0,T]$.
        By Step $1$, there exist $t_{1,2} \in [t^{b_1},t^{b_2}]$ and $t_{2,3} \in [t^{b_2},t^{b_3}]$, such that $\phi_\infty(t_{1,2})$ is $7D$-close to $b_2(\length(b_2)-20R)$ and $\phi_\infty(t_{2,3})$ is $7D$-close to $b_2(\length(b_2)+20R)$. Let $\gamma$ be a segment connecting $\phi_\infty(t_{1,2})$ and $\phi_\infty(t_{2,3})$. By item \ref{item2Lcont} of Lemma \ref{nuovocontinuitylemma} and the fact that $X$ is $(R,D)$-thin, either $\pi_{b_2}(\gamma)$ intersects $I^{b_2}$, or $\gamma$ contains points that are $7D$ close to $b_2(R)$ and $b_2(\length(b_2)-R)$. In the latter case, since $\length(b_2) \geq 200R$, $\gamma$ would not be a segment. Hence, $\pi_{b_2}(\gamma)$ intersects $I^{b_2}$. It follows that $\gamma + {\phi_\infty}_{|S^1_L \setminus [t_{1,2},t_{2,3}]}$ belongs to $A$. In particular, ${\phi_\infty}_{|[t_{1,2},t_{2,3}]}$ is shorter than $\gamma$, so that it is itself a segment.

        At the same time, since ${\phi_\infty}_{|[t^{b_1},t^{b_2}]}$ is a segment, it follows 
        \[
        \pi_{\phi_\infty([t^{b_1},t^{b_2}])}(\phi_\infty(t^{b_2}+3R))=
        \pi_{\phi_\infty([t_{1,2},t^{b_2}])}(\phi_\infty(t^{b_2}+3R)).
        \]
        Combining with the fact that ${\phi_\infty}_{|[t_{1,2},t_{2,3}]}$ is a segment, we obtain the statement.
        
        \textbf{Step 3}: Let $b_1,b_2 \in \{\gamma_1,\gamma_2,l\}$ with $b_1 \neq b_2$. We claim that 
        
        the two arcs of $S^1_L$ defined by
        $S^1_L \setminus [B_{2R}(t^{b_1}) \cup B_{2R}(t^{b_2})]$ are sent by $\phi_\infty$ in different path connected components of $X \setminus [B_{2R}(\phi_\infty(t^{b_1})) \cup B_{2R}(\phi_\infty(t^{b_2}))]$.

        As before, w.l.o.g. we can assume that $t^{b_1} \leq t^{b_2}$. Consider the segment ${\phi_\infty}_{|[t^{b_1},t^{b_2}]}$. We assume by contradiction that there exists a curve $\alpha$ connecting $\phi_\infty(t^{b_1}+3R)$ and $\phi_\infty(t^{b_2}+3R)$ in $X$ which does not intersect $B_{2R}(\phi_\infty(t^{b_1})) \cup B_{2R}(\phi_\infty(t^{b_2}))$. By item \ref{item1Lcont} of Lemma \ref{nuovocontinuitylemma} applied to the segment $\phi_\infty([t^{b_1},t^{b_2}])$ (this can be applied thanks to step $2$), $\pi_{\phi_\infty([t^{b_1},t^{b_2}])}(\alpha)$ contains a $3D$-net of either $[\phi_\infty(t^{b_1}+R),\phi_\infty(t^{b_1}+3R)]$ or $[\phi_\infty(t^{b_1}+3R),\phi_\infty(t^{b_2}-R)]$. In both cases $\alpha$ intersects  $B_{2R}(\phi_\infty(t^{b_1})) \cup B_{2R}(\phi_\infty(t^{b_2}))$, a contradiction.

         \textbf{Step 4.} We claim that $\phi_\infty:S^1_L \to X$ is distance preserving and that $X$ lies in a $D$-neighbourhood of $\phi_\infty(S^1_L)$.

         We show first that $\phi_\infty$ preserves distances. Assume that this is not the case. Then, there exist $x,y \in S^1_L$, and a segment $\bar{l}$, connecting $\phi_\infty(x)$ and $\phi_\infty(y)$, where $\length(\bar{l})$ is less than the length of the shortest arc of $S^1_L$ connecting $x$ and $y$. Let $\eta_1,\eta_2 \subset S^1_L$ be the two arcs connecting $x$ and $y$.
        
        We would like to show that one between the loops $\bar{l}+{\phi_\infty}_{|\eta_1}$ and $\bar{l}+{\phi_\infty}_{|\eta_2}$ belongs to $A$. If this is the case, we obtain a contradiction since $\phi_\infty$ was length minimizing.

        By step $1$, for every $b \in \{\gamma_1,\gamma_2,l\}$, we can assume that $\phi_\infty(t^b)$ is $7D$-close to $b(\length(b)/2)$. This implies that, if a loop has length less than $\length(\gamma_1+\gamma_2+l)$, and it intersects $B_{2R}(\phi_\infty(t^{b}))$ for every $b \in \{\gamma_1,\gamma_2,l\}$, then it admits a reparametrization that belongs to $A$.

        Hence, it suffices to show that one between $\bar{l}+{\phi_\infty}_{|\eta_1}$ and $\bar{l}+{\phi_\infty}_{|\eta_2}$ intersects all the balls $B_{2R}(\phi_\infty(t^{b}))$ for every $b \in \{\gamma_1,\gamma_2,l\}$.

        We consider all the possible cases. If ${\phi_\infty}_{|\eta_1}$ intersects all the balls $B_{2R}(\phi_\infty(t^{b}))$, then there is nothing left to prove. Let now $\{b_1,b_2,b_3\}=\{\gamma_1,\gamma_2,\gamma_3\}$. Assume that ${\phi_\infty}_{|\eta_1}$ intersects only $B_{2R}(\phi_\infty(t^{b_1}))$ and $B_{2R}(\phi_\infty(t^{b_2}))$. By Step $3$, either $\bar{l}$ intersects $B_{2R}(\phi_\infty(t^{b_3}))$, or it intersects both $B_{2R}(\phi_\infty(t^{b_1}))$ and $B_{2R}(\phi_\infty(t^{b_2}))$. In either case, one between $\bar{l}+{\phi_\infty}_{|\eta_1}$ and $\bar{l}+{\phi_\infty}_{|\eta_2}$ intersects all the balls.  The case when ${\phi_\infty}_{|\eta_1}$ intersects exactly one of the aforementioned balls is analogous.
        
        It follows that $\phi_\infty:S^1_L \to X$ preserves distances. Since $2\pi L >2R$, the fact that $X$ is $(R,D)$-thin implies that $X$ is in the $D$-neighbourhood of $\phi_\infty(S^1_L)$. Since $X$ contains a segment of length greater than $200R$, it follows that $L \geq 50R$, concluding the proof.
     \end{proof}
\end{proposition}

The next theorem summarizes the results of this section.

\begin{thm} \label{thm(SNP)}
	Let $(X,\sd)$ be an $(R,D)$-thin metric space. There exists a $1$-dimensional manifold $I$, possibly with boundary, and a distance preserving map $\phi:I \to X$, such that $X$ is contained in the $200R$-neighbourhood of $\phi(I)$. 
	\begin{proof}
	If $X$ is non-compact, the statement follows from Proposition \ref{P10}. Let $X$ be compact. By Proposition \ref{IsometriaCerchio}, either $X$ is contained in the $200R$ neighbourhood of a segment, or it is contained in the $D$-neighbourhood of a loop which is mapped isometrically into $X$. In both cases, the statement holds. 
	\end{proof}
\end{thm}

A consequence of Theorem \ref{thm(SNP)} is that $(R,D)$-thin metric spaces have $1$-Urysohn width bounded from above by a constant $ c(R)>0$ depending only on $R$ (see Corollary \ref{C|Urysohn}).
We recall that the notion of Urysohn width was introduced by Gromov in \cite{FillingGromov,widthGromov,LargeGromov}.

\begin{definition}
    A metric space $(X,\sd)$ has $1$-Urysohn width $\leq d$ if there exists a $1$-simplex $Y$, and a continuous map $f:X \to Y$, such that $\mathrm{diam}(f^{-1}(y)) \leq d$ for every $y \in Y$.
\end{definition}

\begin{corollary} \label{C|Urysohn}
    Let $(X,\sd)$ be an $(R,D)$-thin metric space. Then, there exists $c(R)>0$, a $1$-dimensional connected manifold $Y$, and a continuous map $f:X \to Y$, such that $\mathrm{diam}(f^{-1}(y)) \leq c(R)$ for every $y \in Y$. In particular,    
    $X$ has $1$-Urysohn width $\leq c(R)$.
    \begin{proof}
        Let $\phi:I \to X$ be the map given by Theorem \ref{thm(SNP)}. 

        \textbf{Case 1}: $I=[0,l)$, with $l \in [0,+\infty]$.
        \\Let $f:X \to \bb{R}$ be the distance from $\phi(0)$. Let $x,y \in f^{-1}(t)$ for some $t \geq 0$, and let $t_x,t_y \in [0,l]$ be such that $\phi(t_x)$ and $\phi(t_y)$ are projections of $x$ and $y$ on $\phi(I)$. Since $X$ is in the $200R$-neighbourhood of $\phi(I)$, it holds $t_x,t_y \in [0,l] \cap [t-200R,t+200R]$, so that $\sd(x,y) \leq 1000R$, which implies $\mathrm{diam}(f^{-1}(t)) \leq 1000R$ for every $t \geq 0$.

\textbf{Case 2}: $I=\bb{R}$.
\\Let $B:=B_{1000R}(\phi(0))$ and consider the map $f:X \to \bb{R}$ given by
        \[
        f(x):=
        \begin{cases}
    \sd(B,x) & \text{if } \pi_{\phi(I)}(x) \cap \phi(\bb{R}_+) \neq \emptyset
            \\
            -\sd(B,x) & \text{if } \pi_{\phi(I)}(x) \cap \phi(\bb{R}_-) \neq \emptyset.
         \end{cases}
        \]
        First of all, $f$ is well-defined, since if $\pi_{\phi(I)}(x) \cap \phi(\bb{R}_+) \cap \phi(\bb{R}_-) \neq \emptyset$, then 
        \[
        \pi_{\phi(I)}(x) \cap \phi([-200R,200R]) \neq \emptyset,
        \]
        so that $x \in B$. 
        We now check that $f$ is continuous. This is trivial at points in $\bar{B}$, so let us assume that $x \notin \bar{B}$ and, without loss of generality, that $\pi_{\phi(I)}(x) \cap \phi(\bb{R}_+) \neq \emptyset$. Let $x_i \to x$. Since $x \notin \bar{B}$, $\pi_{\phi(I)}(x) \cap \phi([500R,+\infty)) \neq \emptyset$ so that, for $i$ large enough, $\pi_{\phi(I)}(x_i) \cap \phi(\bb{R}_+) \neq \emptyset$. It follows that $f$ is continuous at $x$, which implies that $f$ is continuous. 
        
        Let now $t \in \bb{R}$ and let $x,y \in f^{-1}(t)$. If $t=0$, then $\sd(x,y) \leq 1000R$. Assume now that $t > 0$, and let 
        $t_x,t_y \in \bb{R}$ be such that $\phi(t_x) \in \pi_{\phi(I)}(x)$ and $\phi(t_y) \in \pi_{\phi(I)}(y)$. Since $t>0$, it follows that $t_x,t_y >0$, so that $t_x,t_y \in [\sd(\phi(0),x)-200R,\sd(\phi(0),x)+200R]$, which implies $\sd(x,y) \leq 1000R$. The case $t<0$ is analogous, concluding the case $I=\bb{R}$.

        \textbf{Case 3}: $I=S^1_L=[0,2\pi L]/\sim$, where $\sim$ identifies the endpoints of $[0,2\pi L]$. 
        \\If $L \leq 10^5R$, the zero map $f:X \to \bb{R}$ defined as $f(x):=0$ gives the desired $1$-Urysohn width estimate. So, we can assume $L \geq 10^5R$.
        Let $B_1:=B_{1000R}(\phi(0))$, $B_2:=B_{1000R}(\phi(\pi L))$, and consider the map $f:X \to [0,2\sd(B_1,B_2)]/\sim$, where $\sim$ identifies the endpoints of $[0,2 \sd(B_1,B_2)]$, given by
        \[
        f(x):=
        \begin{cases}
    \sd(B_1,x) \wedge \sd(B_1,B_2)  & \text{if } \pi_{\phi(I)}(x) \cap  \phi([0,\pi L]) \neq \emptyset \\
            (\sd(B_1,B_2)+\sd(B_2,x)) \wedge  2\sd(B_1,B_2)  & \text{if } \pi_{\phi(I)}(x) \cap  \phi([\pi L,2\pi L]) \neq \emptyset.
         \end{cases}
        \]
        The map $f$ is well-defined, since if $\pi_{\phi(I)}(x) \cap \phi([0,\pi L]) \cap \phi([\pi L,2\pi L]) \neq \emptyset$, then $x \in B_1 \cup B_2$. We now check that $f$ is continuous. This is trivial at points in $B_1 \cup B_2$. So, we consider $x \notin B_1 \cup B_2$ and we assume, without loss of generality, that $\pi_{\phi(I)}(x) \cap \phi([0,\pi L]) \neq \emptyset$. Let $x_i \to x$. Since $x \notin B_1 \cup B_2$, $\pi_{\phi(I)}(x) \cap \phi([500R,\pi L-500R]) \neq \emptyset$, so that, for $i$ large enough, $\pi_{\phi(I)}(x_i) \cap \phi([0,\pi L]) \neq \emptyset$. It follows that $f$ is continuous at $x$, so that $f$ is continuous.

        Let now $t \in [0,2\sd(B_1,B_2))$, and let $x \in f^{-1}(t)$. If $t=0$, then $x \in B_1$ or $\sd(x,B_2)=\sd(B_1,B_2)$. It follows that $\sd(x,B_1) \leq 2000R$, showing that $\mathrm{diam}(f^{-1}(0)) \leq c(R)$. A similar argument holds for $t=\sd(B_1,B_2)$. So, let $t \notin \{0,\sd(B_1,B_2)\}$, and assume without loss of generality that $t \in (0,\sd(B_1,B_2))$.
        Let $x,y \in f^{-1}(t)$, and let 
        $t_x,t_y \in [0,2\pi L)$ be such that $\phi(t_x)$ and $\phi(t_y)$ are projections of $x$ and $y$ on $\phi(I)$. Since $t \in (0,\sd(B_1,B_2))$, it follows that $t_x,t_y \in [0,\pi L]$. Hence, $t_x,t_y \in [\sd(\phi(0),x)-200R,\sd(\phi(0),x)+200R]$, 
        so that $\sd(x,y) \leq 1000R$. This concludes the proof.
    \end{proof}
\end{corollary}

\section{Large scale geometry of almost non-negative Ricci and integrally-positive scalar curvature}\label{Sec:MainGeom}

In this section we draw the main geometric results of the paper, by combining the integral curvature estimates of Section \ref{S3} with the  technical metric results of Section \ref{S1}. 

\subsection{Sufficient curvature conditions for a manifold to be thin}

In this section, we show that manifolds with almost non-negative Ricci curvature and a positive lower integral bound on $\R_2$ (or $\R$, if the manifolds are $3$-dimensional) fit into the framework of Section \ref{S1}. 

The next theorem studies spaces of the form $\bb{R}^{n-3} \times X$ arising as limits of $n$-manifolds with almost non-negative Ricci curvature and positive scalar curvature in integral sense. One should compare also this result with \cite[Theorem 1.1]{wang2023positive}, where limit spaces of the form $\bb{R}^{n-2} \times X$ were studied.

We refer to Definitions \ref{D5} and \eqref{D3} for the notation used in Theorem \ref{thmSNPGeneral}.

\begin{thm} \label{thmSNPGeneral}
    Let $v,\epsilon,L \in (0,+\infty)$, $n,k \in \bb{N}$ with $2 \leq k \leq n$, $s \in (0,1)$ be fixed. There exist $R,D,\delta>0$ with $R \geq 20D$ satisfying the following. Let $(X,\sd,p)$ be a metric space such that $\bb{R}^d \times X$ is a pGH limit of manifolds $(M^n_j,g_j,p_j)$ satisfying one of the following conditions.
    \begin{enumerate}
        \item  $d=k-2$, and
    \begin{equation} \label{Condition1General}
    \ssf{Ric}_{M_j} \geq -\delta, \quad
     \fint_{B_1(x)} \R_{M_j,k} \wedge L \,\de\ssf{Vol}_j  \geq \epsilon, \quad \forall x \in M_j.
    \end{equation}
    \item $n \geq 3$, $d=n-3$, $\ssf{Ric}_{M_j} \geq -\delta$, and
    \begin{equation} \label{condition2General}
    \fint_{B_1(x)} \ssf{Ric}_{M_j,n-2} \wedge 0 \, \de \ssf{Vol}_j  \geq -\delta, \quad  \fint_{B_1(x)} \R_{M_j} \wedge L \,\de\ssf{Vol}_j  \geq \epsilon, \quad \forall x \in M_j.
    \end{equation}
    \item $n \geq 3$, $d=n-3$, and
    \begin{equation} \label{condition3General}
    \ssf{Ric}_{M_j} \geq -\delta, \quad  \fint_{B_1(x)} |\R_{M_j}|^s \,\de\ssf{Vol}_j  \geq \epsilon, \quad \ssf{Vol}_j(B_1(x)) \geq v, \quad \forall x \in M_j.
    \end{equation}
    \end{enumerate}
    Let $r$ be a segment of length greater than $2R$ in $X$. For every $t \in (R,\mathrm{length}(r)-R)$ and $x \in \pi_r^{-1}(r(t))$, it holds $\sd(x,r) < D$, i.e., $(X,\sd)$ is $(R,D)$-thin (see Definition \ref{Condition}).

    In particular, there exists a $1$-dimensional connected manifold $I$ (possibly with boundary) and a distance preserving map $\phi:I \to X$ such that $X$ is contained in the $200R$-neighbourhood of $\phi(I)$.
    \end{thm}
    
    \begin{proof}
        First, consider the condition \eqref{Condition1General}. Suppose by contradiction that the statement fails. 
        Hence, there exist  $\epsilon_0>0$, $R_j,D_j \uparrow + \infty$, $\delta_j \downarrow 0$, metric spaces $(X_j,\sd_j,x_j)$, such that $\bb{R}^{k-2} \times X_j$ arise as limits of manifolds $(M_{i,j})_{i\in \bb{N}}$ satisfying 
         \begin{equation} \label{Condition1Pf}
    \ssf{Ric}_{M_{i,j}} \geq -\delta_j, \quad
     \fint_{B_1(x)} \R_{M_{i,j},k} \wedge L \,\de\ssf{Vol}_{i,j}  \geq \epsilon_0, \quad \forall x \in M_{i,j},
    \end{equation}
         and segments $r_j \subset X_j$ with $\length(r_j) \geq 2R_j$, such that $\sd_j(q_j,r_j) \geq D_j$ for some $q_j \in \pi_{r_j}^{-1}(r_j(R_j))$.
    In particular, there exist points $q_j$, whose distance from $r_j$ is greater than $D_j$, and whose point realizing the distance from $r_j$ is $r_j(R_j)$. Let $\gamma_j$ be the segment connecting $q_j$ and $r_j(R_j)$.

    By Gromov pre-compactness theorem, up to subsequences, the sequence  $(X_j,\sd_j,r_j(R_j))$ converges in pGH-sense to a Ricci limit space $(X,\sd,p)$. Moreover, there exists a measure $\m \in \mathcal{M}(X)$ -- arising as a limit of the renormalized volume measures of $M_{i,j}$ -- such that $(X,\sd,\m,p)$ is an $\RCD(0,n)$ space.
    
    In addition, $X$ contains a line which arises as a limit of the segments $r_j$ centered in $r_j(R_j)$, since $R_j \uparrow \infty$. 
    Hence, by the splitting theorem \cite{Gigli-Split}, $X$ splits isomorphically as a metric measure space as  $X=Y \times \bb{R}$  for an $\RCD(0,n-1)$ space $(Y,\sd_y,\m_y)$. Moreover, $Y$ cannot be compact since the segments $\gamma_j$, which realize the distance from $r_j$, have arbitrarily large length. 
    In particular, $Y$ contains itself a ray.
    Hence, repeating the argument above, there exists a sequence $p_j' \in Y$, such that $(Y,\sd_y,p_j')$ converges in pGH-sense to a metric space $\bb{R} \times Y'$. In particular, there exists a sequence $p_j'' \in X_j$ such that $(X_j,\sd_j,p_j'')$ converges in pGH-sense to $\bb{R}^2 \times Y'$. Applying  Theorem \ref{CT1} to $M_{i,j}$ for $i$ and $j$ large enough,  we arrive to a contradiction, since $\delta_j\downarrow 0$ and $\epsilon_0>0$ is fixed in \eqref{Condition1Pf}. 

    If we assume conditions \eqref{condition2General} or \eqref{condition3General}, the proof is analogous, by using Theorem \ref{CT2} instead of Theorem \ref{CT1}.

    Finally, by Theorem \ref{thm(SNP)}
    , there exists a $1$-dimensional connected manifold $I$ (possibly with boundary) and a distance preserving map $\phi:I \to X$ such that $X$ is contained in the $200R$-neighbourhood of $\phi(I)$.
    \end{proof}

As an immediate consequence of the previous theorem, we deduce the following key corollary.

\begin{corollary} \label{CorollarySNP}
    Let $v,\epsilon,L \in (0,+\infty)$, $n \in \bb{N}$, $s \in (0,1)$ be fixed. There exist $R,D,\delta>0$ with $R \geq 20D$ satisfying the following. Let $(X,\sd,p)$ be a pGH limit of manifolds $(M^n_j,g_j,p_j)$ satisfying one of the following conditions.
    \begin{enumerate}
        \item  $n\geq 2$ and
    \begin{equation} \label{Condition1}
    \ssf{Ric}_{M_j} \geq -\delta, \quad
     \fint_{B_1(x)} \R_{M_j,2} \wedge L \,\de\ssf{Vol}_j  \geq \epsilon, \quad \forall x \in M_j.
    \end{equation}
    \item $n=3$, $\ssf{Ric}_{M_j} \geq -\delta$ and
    \begin{equation} \label{condition2}
    \fint_{B_1(x)} \ssf{Sec}_{M_j} \wedge 0 \, \de \ssf{Vol}_j  \geq -\delta, \quad  \fint_{B_1(x)} \R_{M_j} \wedge L \,\de\ssf{Vol}_j  \geq \epsilon, \quad \forall x \in M_j.
    \end{equation}
    \item $n=3$ and
    \begin{equation} \label{condition3}
    \ssf{Ric}_{M_j} \geq -\delta, \quad  \fint_{B_1(x)} |\R_{M_j}|^s \,\de\ssf{Vol}_j  \geq \epsilon, \quad \ssf{Vol}_j(B_1(x)) \geq v, \quad \forall x \in M_j.
    \end{equation}
    \end{enumerate}
    Let $r$ be a segment of length greater than $2R$ in $X$. For every $t \in (R,\mathrm{length}(r)-R)$ and $x \in \pi_r^{-1}(r(t))$, it holds $\sd(x,r) < D$, i.e., $(X,\sd)$ is $(R,D)$-thin (see Definition \ref{Condition}).

    In particular, there exists a $1$-dimensional connected manifold $I$ (possibly with boundary) and a distance preserving map $\phi:I \to X$ such that $X$ is contained in the $200R$-neighbourhood of $\phi(I)$.
    \end{corollary}

The remaining part of the section is devoted to showing that if a manifold satisfies one of the conditions \eqref{Condition1General}, \eqref{condition2General}, \eqref{condition3General}, so do its coverings (modulo a rescaling by dimensional constants). This is the content of Propositions \ref{Prop:CoverInt} and \ref{PC3}. The proof of the next proposition follows from \cite[Lemma 1.6]{kapovitch2011structurefundamentalgroupsmanifolds}; we report it for completeness of presentation.

\begin{proposition}\label{Prop:CoverInt}
    Let $n \in \bb{N}$ and let $K \in \bb{R}$. Let $(M^n,g,p)$ be a pointed Riemannian manifold with $\Ric_M \geq K$. Let $(\tilde{M}^n,\tilde{g},\tilde{p})$ be the universal covering of $M$, with covering map $\pi:\tilde{M} \to M$ such that $\pi(\tilde{p})=p$. Let $F:M \to \bb{R}_+$ be a measurable locally bounded function and let $\tilde{F}:=F \circ \pi$. 
    \begin{enumerate}
        \item \label{item1Covering}
        It holds
        \[
\fint_{\tilde{B}_3(\tilde{p})} 
    \tilde{F} \, \de\widetilde{\ssf{Vol}} \geq c(K,n) \fint_{B_1(p)} F \, \de\ssf{Vol}.
    \]
    \item \label{item2Covering}
    There exists $p' \in M$ such that
    \[
\fint_{\tilde{B}_3(\tilde{p})} 
    \tilde{F} \, \de\widetilde{\ssf{Vol}} \leq c(K,n) \fint_{B_1(p')} F \, \de\ssf{Vol}.
    \]
    \end{enumerate}
    \begin{proof}
Consider a measurable function $j:B_1(p) \to \tilde{B}_1(\tilde{p})$ such that $\pi(j(x))=x$ for every $x \in B_1(p)$. Let $T:=j(B_1(p))$, so that $\mathrm{diam}(T) \leq 2$ and 
\[
\fint_T \tilde{F} \, d\widetilde{\ssf{Vol}}=\fint_{B_1(p)}F \, d\ssf{Vol}.
\]
Let $S \subset \tilde{M}$ be the union of the sets $g(T)$ for every deck transformation $g:\tilde{M} \to \tilde{M}$ of $\pi:\tilde{M} \to M$ such that $g(T) \cap \tilde{B}_1(\tilde{p}) \neq \emptyset$. It follows that
\begin{equation} \label{E|KW1}
\tilde{B}_1(\tilde{p}) \subset S \subset \tilde{B}_3(\tilde{p}) \quad \text{and} \quad
\fint_S \tilde{F} \, d\widetilde{\ssf{Vol}}=\fint_{B_1(p)}F \, d\ssf{Vol}.
\end{equation}

\textbf{Proof of }\ref{item1Covering}. Using that $\widetilde{\ssf{Vol}}(\tilde{B}_3(\tilde{p})) \leq c(K,n) \widetilde{\ssf{Vol}}(\tilde{B}_1(\tilde{p}))$ and \eqref{E|KW1}, it holds
\[
\fint_{B_1(p)}F \, d\ssf{Vol} \leq \frac{\widetilde{\ssf{Vol}}(\tilde{B}_3(\tilde{p})) }{\widetilde{\ssf{Vol}}(\tilde{B}_1(\tilde{p}))}\fint_{\tilde{B}_3(\tilde{p})} \tilde{F} \, d\widetilde{\ssf{Vol}}
\leq c(K,n) \fint_{\tilde{B}_3(\tilde{p})} \tilde{F} \, d\widetilde{\ssf{Vol}}.
\]

\textbf{Proof of }\ref{item2Covering}. Arguing as for \eqref{E|KW1}, we find $S' \subset \tilde{M}$ such that
\begin{equation} \label{E|KW2}
\tilde{B}_3(\tilde{p}) \subset S' \subset \tilde{B}_9(\tilde{p}) \quad \text{and} \quad
\fint_{S'} \tilde{F} \, d\widetilde{\ssf{Vol}}=\fint_{B_3(p)}F \, d\ssf{Vol}.
\end{equation}
Let $\{B_1(p_i)\}_{i \in I}$ be a covering of $B_3(p)$ with $I \subset \bb{N}$, such that $\{B_{1/5}(p_i)\}_{i \in I}$ are disjoint. By \eqref{E|KW2}, it holds
\begin{equation} \label{E|KW3}
\fint_{\tilde{B}_3(\tilde{p})} \tilde{F} \, d\widetilde{\ssf{Vol}} \leq 
\frac{\widetilde{\ssf{Vol}}(\tilde{B}_9(\tilde{p})) }{\widetilde{\ssf{Vol}}(\tilde{B}_3(\tilde{p}))}
\fint_{B_3(p)}F \, d\ssf{Vol} \leq
c(K,n) \sum_{i \in I} \fint_{B_1(p_i)}F \, d\ssf{Vol}.
\end{equation}
By Bishop-Gromov's inequality, it holds $\# I \leq c(K,n)$. Setting $p' \in \{p_i\}_{i \in I}$ such that
\[
\fint_{B_1(p')}F \, d\ssf{Vol}=\max_{i \in I}\fint_{B_1(p_i)}F \, d\ssf{Vol}, 
\]
and combining with \eqref{E|KW3}, the statement follows.
    \end{proof}
\end{proposition}

Proposition \ref{Prop:CoverInt}, together with the fact that covering maps reduce volumes of balls, imply the following proposition.

\begin{proposition} \label{PC3}
    Let $v,\epsilon,\delta,L \in (0,+\infty)$, $n \in \bb{N}$, $s \in (0,1)$ be fixed. Let $(M^n,g,p)$ be a manifold satisfying one of the conditions \eqref{Condition1General}, \eqref{condition2General}, \eqref{condition3General}. Let $(\tilde{M}^n,\tilde{g},\tilde{p})$ be the universal covering, with covering map $\pi:\tilde{M} \to M$ such that $\pi(\tilde{p})=p$. Then $(\tilde{M},\tilde{g}/3^2,\tilde{p})$ satisfies the corresponding condition among \eqref{Condition1General}, \eqref{condition2General}, \eqref{condition3General}, where the constants have been rescaled by factors depending only on the dimension $n$.
\end{proposition}

\begin{corollary} \label{CC4}
    Let $v,\epsilon,L \in (0,+\infty)$, $n \in \bb{N}$, $s \in (0,1)$ be fixed. There exist $R,D,\delta>0$ with $R \geq 20D$ satisfying the following. Let $(M^n,g)$ be a manifold satisfying one of the  conditions \eqref{Condition1}, \eqref{condition2}, \eqref{condition3}. Let $(\tilde{M},\tilde{g},\tilde{p})$ be the universal covering of $M$ and let $\pi:\tilde{M} \to M$ be the covering map. Then, both $(M,g)$ and $(\tilde{M},\tilde{g})$ are $(R,D)$-thin (see Definition \ref{Condition}).
    \begin{proof}
        The statement follows combining Proposition \ref{PC3} and Corollary \ref{CorollarySNP}.
    \end{proof}
\end{corollary}

\subsection{Large scale geometry of thin manifolds}

In this section, relying on Theorem \ref{thm(SNP)}, we prove metric and topological properties of manifolds satisfying one of the conditions \eqref{Condition1}, \eqref{condition2}, \eqref{condition3}. The next result is an immediate consequence of Theorem \ref{thm(SNP)} and Corollary \ref{CorollarySNP}.

\begin{thm} \label{thmCentral}
Let $v,\epsilon,L \in (0,+\infty)$, $n \in \bb{N}$, $s \in (0,1)$ be fixed.
    There exist $C,\delta>0$ with the following property. Let $(X,\sd,p)$ be a limit of pointed Riemannian manifolds satisfying one of the conditions \eqref{Condition1}, \eqref{condition2}, \eqref{condition3}. Then, there exists a $1$-dimensional connected manifold $I$ (possibly with boundary) and a distance preserving map $\phi:I \to X$ such that $X$ is contained in the $C$-neighbourhood of $\phi(I)$.
    \end{thm}

    \begin{remark} \label{R6}
        The previous result fails if we remove the lower bounds on the Ricci curvature in conditions \eqref{Condition1} and \eqref{condition3}, thanks to the gluing results in \cite{GromovLawsonSurgery1, GromovLawsonSurgery2,WolfsonKricci}. Indeed, consider the gluing of four copies of $S^2 \times [0,+\infty)$ to $S^3$ minus four disjoint small balls. By \cite{WolfsonKricci}, we can equip this space with a Riemannian metric $g$ with the following properties:
        \begin{enumerate}
        \item $g$ coincides with the metric on $S^3$ and with the metric on each of the branches $S^2 \times [0,+\infty)$ outside of a small neighbourhood of the surgery.
        \item $g$ has $\R_2 \geq 0$.
        \end{enumerate}
        Hence, there exists $\epsilon>0$ such that the glued manifold $M$ satisfies
        \[
        \fint_{B_1(x)} \R_2 \, d\ssf{Vol} \geq \epsilon \quad \text{for every } x \in M.
        \]
        At the same time, Theorem \ref{thmCentral} fails on $M$. With a similar construction using \cite{GromovLawsonSurgery1, GromovLawsonSurgery2}, one shows that also the Ricci bound in \eqref{condition3} cannot be removed.
    \end{remark}
    
 The next result concerns volumes of manifolds satisfying one of the conditions \eqref{Condition1}, \eqref{condition2}, or \eqref{condition3}.

\begin{thm} \label{thmVolume}
    Let $v,\epsilon,L \in (0,+\infty)$, $n \in \bb{N}$, $s \in (0,1)$ be fixed.
    There exist $C,\delta>0$ satisfying the following. Let $(M^n,g)$ be a manifold satisfying one of the conditions \eqref{Condition1}, \eqref{condition2}, or \eqref{condition3}. Then 
    \begin{enumerate}
    \item \label{item1thmVolume}
        $\sup_{x \in M} \ssf{Vol}(B_t(x)) \leq Ct$, for all $t>0$.
       \item \label{item2thmVolume}
       If $\Ric_M \geq 0$, then $\inf_{x \in M} \ssf{Vol}(B_1(x))>0$.
       \end{enumerate}
    \begin{proof}
    By Theorem \ref{thmCentral}, there exists a $1$-dimensional connected manifold $I$ (possibly with boundary), a distance preserving map $\phi:I \to X$, and a constant $C' >1$, such that $X$ is contained in the $C'$-neighbourhood of $\phi(I)$. 
    
        \textbf{Proof of} \ref{item1thmVolume}. 
        For $t \in (0,1]$, the claim is a consequence of the classical volume comparison results for manifolds with Ricci curvature bounded from below by $-1$. Hence, we only consider the case $t>1$.
        
        Let $x \in X$ and let $\phi(s_x) \in \pi_{\phi(I)}(x)$. Let $\{s_i\}_{i=1}^{\lceil t \rceil}$ be a $1$-net of $B^I_t(s_x)$. 
        By triangular inequality, $B_t(x)$ is contained in the $10C'$-neighbourhood of $\phi(B^I_t(s_x))$. Hence, again by triangle inequality and our choice of $\{s_i\}_{i=1}^{\lceil t \rceil}$, it follows
        \[
        B_t(x) \subset \bigcup_{i=1}^{\lceil t \rceil} B_{20C'}(\phi(s_i)).
        \]
        By the aforementioned volume bounds for manifolds with Ricci curvature bounded from below by $-1$, there exists a constant $C>0$ (depending only on the dimension), such that $\ssf{Vol}(B_t(x)) \leq C \lceil t \rceil \leq 2Ct$.

    \textbf{Proof of} \ref{item2thmVolume}.
    If the manifold is $M$ is compact, there is nothing to prove. If $M$ is non-compat and it contains a line, then $M=\bb{R} \times N$, where $N$ is a compact manifold. Hence, also in this case the statement follows.
    
    So we only need to consider the case when $M$ is non-compact and it is contained in the $C'$ neighbourhood of a ray.
    Fix $x\in M$ and let $t_x \geq 0$ be such that  $ x\in B_{C'}(r(t_x))$.
    By Brunn-Minkowski inequality \cite{S2}, 
    \begin{equation} \label{E26}
    \ssf{Vol}(A_{1/2}(2t_x)) \geq (1/2)^n \; \ssf{Vol}(B_1(r(0))),
    \end{equation}
    where $A_{1/2}(2t_x)$ is the set consisting of the intermediate points of segments connecting points in $B_1(r(0))$ and in $B_1(r(2t_x))$. 
    
    We claim that $A_{1/2}(2t_x)$ is contained in $B_{10C'}(r(t_x))$. Assume for the moment that the claim holds. In this case, by the doubling property and \eqref{E26}, it holds $\ssf{Vol}(B_1(x)) \geq v'>0$ for some $v'$ depending on $\ssf{Vol}(B_1(r(0))), n$ and $C'$. 

    To conclude the proof, we are left to show that $A_{1/2}(2t_x)$ is contained in $B_{10C'}(r(t_x))$.
    To this aim, consider a segment $l$ connecting points $a \in B_1(r(0))$ and $b \in B_1(r(2t_x))$. Let $t_l \geq 0$ be such that $r(t_l) \in \pi_r\big(l(\length(l)/2)\big)$ on $r$. Since $X$ is in the $C'$-neighbourhood of $r$, by triangle inequality, it holds
    \[
    \sd(r(0),r(t_l)) \leq 1+C'+\length(l)/2, \quad
    \sd(r(t_x),r(t_l)) \leq 1+C'+\length(l)/2.
    \]
    Since $\length(l) \leq t_x+2$ again by triangle inequality, it follows that $t_l \in [t_x/2-4C',t_x/2+4C']$. Since $l(\length(l)/2)$ is $C'$-close to $r(t_l)$, it follows $l(\length(l)/2) \in B_{10C'}(r(t_x))$, as claimed.
    \end{proof}
\end{thm}

The previous result, in particular, shows that the manifolds in question, when $\Ric_M \geq 0$, fit into the framework of \cite{ZhuLinear}.

\begin{remark} \label{R7}
        An analog of Theorem \ref{thmVolume} is not to be expected in higher dimension. Indeed a manifold with $\R_3 \geq 1$ and almost non-negative Ricci curvature might have exponential volume growth at infinity (consider rescalings of products of the hyperbolic plane and a sufficiently small $2$-sphere). Similarly, there exists a manifold with $\R_5 \geq 1$ and non-negative Ricci curvature with unit balls of arbitrarily small volume (consider a product of a $2$-sphere and the manifold constructed in \cite[Example $26$]{SormaniExample}).
    \end{remark}

Let us recall some  properties of covering spaces.
Let $\pi:\tilde{X} \to X$ be a covering between metric spaces. Let $p \in X$ and $\tilde{p} \in \pi^{-1}(p)$. Given a loop $\gamma$ in $p$, we can consider its lift $\tilde{\gamma}$ on $\tilde{X}$ with starting point $\tilde{p}$. 
This lift is unique, and homotopic loops give rise to homotopic lifts (see \cite[p.\,60]{Hatcher}).

In particular, one can consider the endpoint of $\tilde{\gamma}$, which we will denote $M(\tilde{p},\gamma)$. This way of obtaining new points, given an element $[\gamma] \in \pi_1(X)$ 
and a point $\tilde{p} \in \tilde{X}$, is called the \emph{monodromy action} of $\gamma$ on $\tilde{p}$. 
Since homotopic loops give rise to the same monodromy action, if $[\gamma]$ is the identity in $\pi_1(X)$, then $M(\tilde{p},\gamma)=\tilde{p}$.

\begin{thm} \label{thmTopology}
	Let $v,\epsilon,L \in (0,+\infty)$, $n \in \bb{N}$, $s \in (0,1)$ be fixed. There exist $C,\delta>0$ satisfying the following.
    Let $(M^n,g)$ be a manifold satisfying one of the conditions \eqref{Condition1}, \eqref{condition2}, or \eqref{condition3}. The following hold.
    \begin{enumerate}
    \item \label{item0thmTopology} $M$ has $1$-Urysohn width $\leq C$.
    	\item \label{item1thmTopology}$M$ has at most two ends.
    	\item \label{item2thmTopology}$\mathrm{b}_1(M) \leq 1$.
    	\item \label{item3thmTopology}$\pi_1(M)$ is infinite if and only if $M$ is compact and its universal cover is non-compact.
    	\item \label{item4thmTopology}If there exists a loop $\gamma \subset M$ such that $[\gamma] \in \pi_1(M)$ has infinite order, then $M$ is contained in a $C$-neighbourhood of $\gamma$.
    \end{enumerate}
    \begin{proof}
    \textbf{Proof of} \ref{item0thmTopology}. The claim follows by Corollary \ref{CorollarySNP} and Corollary \ref{C|Urysohn}.
    
    \textbf{Proof of} \ref{item1thmTopology}.
    Let $C>0$ be the constant given by Theorem \ref{thmCentral}.
        If $X$ has more than one end, then it is contained in a $C$-neighbourhood of a line by Theorem \ref{thmCentral} and cannot have a third end.
        
        \textbf{Proof of} \ref{item2thmTopology}.
        Let $\tilde{M}$ the universal cover of $M$. 
        By Proposition \ref{PC3} and Theorem \ref{thmVolume}, the volume of balls in $\tilde{M}$ grows at most linearly, so that by Proposition \ref{P|anderson}, it holds $\mathrm{b}_1(M) \leq 1$. 
        
        \textbf{Proof of} \ref{item3thmTopology}.
        It is clear that if $M$ is compact and its universal cover $\tilde{M}$ is non-compact, then $\pi_1(M)$ is infinite.
        We next show the converse.
        Since $\pi_1(M)$ has infinitely many elements, the universal cover $\tilde{M}$ is not compact. Hence, it suffices to show that $M$ is compact.
        
        Assume by contradiction that $M$ is not compact. We denote by $\sd$ and $\tilde{\sd}$ respectively distances in $M$ and $\tilde{M}$. Let $C>0$ be the constant given by Theorem \ref{thmCentral}, and let $r_M \subset M$ and $r_{\tilde{M}} \subset \tilde{M}$ be the rays or lines whose $C$-neighbourhoods contain respectively $M$ and $\tilde{M}$. 
        
        Let $p \in \tilde{M}$ be a preimage of $r_M(0)$ in $\tilde{M}$ via the covering map.
        Since $\pi_1(M)$ has infinitely many elements, there exists a sequence $\{\gamma_i\}_{i \in \bb{N}} \subset \pi_1(M)$ such that $\tilde{\sd}(M(p,\gamma_i),p) \to + \infty$.
        Combining this with the fact that $\tilde{M}$ is contained in the $C$-neighbourhood of $r_{\tilde{M}}$, we get that there exist $t_0,t_1,t_2$ in the domain of $r_{\tilde{M}}$ (i.e.,\ $\bb{R}$ if $r_{\tilde{M}}$ is a line and $\bb{R}_+$ if $r_{\tilde{M}}$ is a ray) such that $t_2 \geq 100C+t_1 \geq t_0+200C $, and $\gamma_1,\gamma_2 \in \pi_1(M)$ with
        \begin{equation} \label{E25}
        \begin{split}
        % {\color{red}\tilde{\sd}(p,r_{\tilde{M}}(t_0))+
        % \tilde{\sd}(M(p,\gamma_1),r_{\tilde{M}}(t_1))+\tilde{\sd}(M(p,\gamma_2),r_{\tilde{M}}(t_2)) \leq 6C}\\
        \tilde{\sd}(p,r_{\tilde{M}}(t_0))\leq C, \quad
        \tilde{\sd}(M(p,\gamma_1),r_{\tilde{M}}(t_1)) \leq C, \quad \tilde{\sd}(M(p,\gamma_2),r_{\tilde{M}}(t_2)) \leq C.
        \end{split}
        \end{equation}
         
         Consider a lift $\tilde{r}_M$ of $r_M$ to $\tilde{M}$ such that 
         $\tilde{r}_M(0)=M(p,\gamma_1)$. Clearly, $\tilde{r}_M$ is a ray in $\tilde{M}$. By \eqref{E25}, Lemma \ref{nuovocontinuitylemma}, and our choice of $t_0,t_1,t_2$, there exists $t\geq 10C$ such that $\tilde{r}_M(t)$ is $2C$-close to either $r_{\tilde{M}}(t_0)$ or $r_{\tilde{M}}(t_2)$. Assume without loss of generality that
         \begin{equation}\label{eq:trXtrTt0<C}
         \tilde{\sd}(\tilde{r}_M(t), r_{\tilde{M}}(t_0))\leq 2C.
         \end{equation}
         
         Recall that each $\gamma \in \pi_1(M)$ induces an isometry $\phi$ of $\tilde{M}$ defined by $\phi(\tilde{x})=M(\tilde{x},\gamma)$ for every $\tilde{x} \in \tilde{M}$, and that $M$ is isometric to the quotient of $\tilde{M}$ by such isometries.
         Hence: 
         \[
         10C \leq t = 
         {\sd}(r_M(0),r_M(t)) \leq \tilde{\sd}(p,\tilde{r}_M(t))+\tilde{\sd}(M(p,\gamma_1),\tilde{r}_M(0))=\tilde{\sd}(p,\tilde{r}_M(t)).
         % \sd(r_X(t),r_M(0)) 
         \]
         Since, by \eqref{E25} and \eqref{eq:trXtrTt0<C}, 
         \[
        \tilde{\sd}(p,\tilde{r}_M(t))\leq\tilde{\sd}(p,r_{\tilde{M}}(t_0))+\tilde{\sd}(r_{\tilde{M}}(t_0), \tilde{r}_M(t)) \leq 3C
        \]
          the combination of the last two inequalities yields a contradiction. Hence, $M$ is compact.
          
          \textbf{Proof of} \ref{item4thmTopology}. Let $\gamma \subset M$ be a loop such that $[\gamma] \in \pi_1(M)$ has infinite period. Let $\tilde{\gamma}$ be a lift of $\gamma$ to $\tilde{M}$, and let $\tilde{\gamma}^{\bb{N}}:[0,+\infty) \to \tilde{M}$ be a lift of the infinite curve $\sum_{i \in \bb{N}} \gamma$. Since $[\gamma] \in \pi_1(M)$ has infinite period, it holds
          \begin{equation} \label{eqtop}
          \limsup_{t \to + \infty} \tilde{\sd}(\tilde{\gamma}^{\bb{N}}(0),\tilde{\gamma}^{\bb{N}}(t))= + \infty.
          \end{equation}
          Let $r_{\tilde{M}}$ be the ray or the line in $\tilde{M}$ whose $C'$-neighbourhood contains $\tilde{M}$ (which exists combining Theorem \ref{thmCentral} and Proposition \ref{PC3}). By \eqref{eqtop} and Lemma \ref{nuovocontinuitylemma}, there exist $t_1,t_2$ in the domain of $r_{\tilde{M}}$ with $t_2 \geq 100C'+100\length(\gamma)+t_1$ such that $\pi_{r_{\tilde{M}}}(\tilde{\gamma}^{\bb{N}})$ contains a $3C'$-net of $[r_{\tilde{M}}(t_1),r_{\tilde{M}}(t_2)]$. Assume now that there is a point $p \in M$ such that $\sd(p,\gamma) \geq 10C'$. Let $l$ be a segment realizing the distance from $p$ to $\gamma$ and let $p_l$ be the footpoint of $l$ on $\gamma$. The set of preimages of $p_l$ in $\tilde{M}$ via the covering map form a $\length(\gamma)$-net of $\tilde{\gamma}^{\bb{N}}$ in $\tilde{M}$. In particular, there is a preimage $\tilde{p}_l$ which is $10C'+\length(\gamma)$ close to $r_{\tilde{M}}((t_1+t_2)/2)$. Let $\tilde{l}$ be the lift of $l$ to $\tilde{M}$ starting from $\tilde{p}_l$ (we are orienting $l$ from $p_l$ to $p$). Then,
          \begin{equation} \label{eqtop2}
          \tilde{\sd}(\tilde{p}_l,\tilde{\gamma}^{\bb{N}})=\tilde{\sd}(\tilde{l}(\length(\tilde{l}),\tilde{\gamma}^{\bb{N}})=\tilde{\sd}(l(\length(l),\gamma) \geq 10C'.
          \end{equation}
          At the same time, since $\tilde{p}_l$ is $10C'+\length(\gamma)$ close to $r_{\tilde{M}}((t_1+t_2)/2)$, it holds $r_{\tilde{M}}(t) \in \pi_{r_{\tilde{M}}}\big( \tilde{p}_l \big)$ for some $t \in[t_1,t_2]$. Using that $r_{\tilde{M}}(t)$ is $7C'$-close to a point in $\tilde{\gamma}^{\bb{N}}$ and \eqref{eqtop2}, it holds
          \[
          \tilde{\sd}(\tilde{p}_l,r_{\tilde{M}})=\tilde{\sd}(\tilde{p}_l,r_{\tilde{M}}(t)) \geq \tilde{\sd}(\tilde{p}_l,\tilde{\gamma}^{\bb{N}})-7C' \geq 3C'.
          \]
          This is a contradiction.
    \end{proof}
\end{thm}

\section*{Appendix}

The appendix shows how to obtain a result by Jiang-Naber \cite[Theorem 2.17]{Naberconjectures}. Since the result was claimed in \cite{Naberconjectures} without proof, for the reader's convenience,  we provide an argument consisting in a combination of   Theorem \ref{CT1}  with \cite[Theorem 1.3]{CheegerNaber}. The result is not used anywhere else in the paper, and it is included as a further application of the techniques developed in the manuscript.

\begin{definition}
    Let $(M^n,g,x)$ be a pointed Riemannian manifold. The quantitative singular strata are defined, for $r \in (0,1)$, as
    \[
    \mathcal{S}_{\delta,r}^k:=\{y \in B_1(x): B_s(y) \text{ is not } (k+1,\delta) \text{-symmetric, for any } s \in (r,1)\},
    \]
    \[
    \mathcal{S}_{\delta}^k:=\{y \in B_1(x): B_s(y) \text{ is not } (k+1,\delta) \text{-symmetric for any } s \in (0,1)\}.
    \]
    We use the abbreviated notation $\mathcal{S}_{\delta,r}$ for $\mathcal{S}^{n-1}_{\delta,r}$, and $\mathcal{S}_{\delta}$ for $\mathcal{S}^{n-1}_{\delta}$. The set $B_1(x) \setminus \mathcal{S}_{\delta,r}$ is denoted by $\mathcal{R}_{\delta,r}$.
\end{definition}

\begin{proposition} \label{P2}
    Let $K \in \bb{R}$ and $n \in \bb{N}$ be fixed and let $(M^n,g,x)$ be a manifold with $\Ric_M \geq K$. For every $\delta>0$ there exists $\eta(K,n,\delta)>0$ such that
    $
    \mathcal{S}_{\delta,r} \subset \mathcal{S}^{n-2}_{\eta,r}
    $,
    for every $r \in (0,1]$.
    \begin{proof}
        The statement follows arguing by contradiction and using Theorem \ref{T2}.
    \end{proof}
\end{proposition}

The next theorem corresponds to \cite[Theorem 1.3]{CheegerNaber}.
\begin{thm}[Cheeger-Naber] \label{CheegerNaberMain}
    Let $K \in \bb{R}$, $n \in \bb{N}$, $v,\delta,\eta >0$ be fixed and let $(M^n,g,x)$ be a manifold with $\Ric_M \geq K$ and $\ssf{Vol}(B_1(x)) \geq v$. Denoting by $T_r(\mathcal{S}^k_{\delta,r})$ the $r$-neighbourhood of $\mathcal{S}^k_{\delta,r}$, there exists  $C=C(n,v,\delta,\eta)>0$ such that
    \[
    \ssf{Vol}(T_r(\mathcal{S}^k_{\delta,r})) \leq Cr^{n-k-\eta}.
    \]
\end{thm}

\begin{thm}\label{thm:IntRsB1Riccidelta}
    Let $n \in \bb{N}$, $v>0$, $s \in (0,1)$ be fixed. There exist $\delta(n,s),C(n,v,s)>0$ such that, if $(M^n,g,x)$ has $\Ric_M \geq -\delta$ and $\ssf{Vol}(B_1(x)) \geq v$, then
    \[
    \fint_{B_{1}(x)} |\R|^s \,\de\ssf{Vol} \leq  C(n,v,s).
    \]
    \begin{proof}
        By Theorem \ref{CT1} and a standard rescaling, there exists $\delta(n,s)>0$ such that if $(M,g,x)$ has $\Ric_M \geq -\delta$, then
        any $\delta$-regular ball $B_{10r}(y) \subset M$ with $y \in B_1(x)$ and $r \in (0,1]$ satisfies
     \begin{equation}\label{eq;IntRsRegBall}
     \fint_{B_r(y)} |\R|^s \, \de \ssf{Vol} \leq r^{-2s}.
     \end{equation}
     Without loss of generality, for the rest of the proof we will assume that $\delta\leq (n-1)$.
    
    Let $\alpha\in (0,1)$ be fixed,  and consider the disjoint union
    \[
    B_1(x)=\mathcal{R}_{\delta,\alpha} \cup \bigcup_{k \in \bb{N}} (\mathcal{S}_{\delta,\alpha^{k}} \setminus \mathcal{S}_{\delta,\alpha^{k+1}}).
    \]
     Consider a covering $\{B_{r_j/10}(x_j)\}_{j \in \bb{N}}$ of $B_1(x)$ such that each $B_{r_j}(x_j)$ is $\delta$-regular and satisfies the following.
     \begin{enumerate}
         \item \label{III1} If $x_j \in \mathcal{R}_{\delta,\alpha}$, then $r_j \in [\alpha,1)$.
         \item If $x_j \in \mathcal{S}_{\delta,\alpha^{k}} \setminus \mathcal{S}_{\delta,\alpha^{k+1}}$, then $r_j \in [\alpha^{k+1},\alpha^k]$.
     \end{enumerate}
     By refining the covering, we may assume that $\{B_{r_j/50}(x_j)\}_{j \in \bb{N}}$ consists of disjoint balls.
     We will bound
    \[
    \sum_{x_j \in \mathcal{R}_{\delta,\alpha} }
    \int_{B_{r_j/10}(x_j)} |\R|^s \,\de\ssf{Vol}
    \quad
    \text{and}
    \quad 
     \sum_{k \in \bb{N}} \, \sum_{x_j \in \mathcal{S}_{\delta,\alpha^{k}} \setminus \mathcal{S}_{\delta,\alpha^{k+1}}}
    \int_{B_{r_j/10}(x_j)} |\R|^s \,\de\ssf{Vol}
    \]
    separately. 
    
    We treat the case $x_j \in \mathcal{R}_{\delta,\alpha}$ first.
    Combining the assumption that $0<v \leq \ssf{Vol}(B_1(x)) \leq C(n)$ and condition \ref{III1} with Bishop-Gromov's volume monotonicity, we deduce that the number of balls in the aforementioned covering such that $x_j \in \mathcal{R}_{\delta,\alpha}$ is bounded above by $c(n,v,s)$. Hence,  \eqref{eq;IntRsRegBall}  yields
    \begin{equation} \label{E14}
    \sum_{x_j \in \mathcal{R}_{\delta,\alpha} }
    \int_{B_{r_j/10}(x_j)} |\R|^s \,\de\ssf{Vol} \leq 10^{2s}\alpha^{-2s} \sum_{x_j \in \mathcal{R}_{\delta,\alpha} }
    \ssf{Vol}({B_{r_j/10}(x_j)}) \leq  c(n,v,s).
    \end{equation}
    
    We now consider the case $x_j \in \mathcal{S}_{\delta,\alpha^{k}} \setminus \mathcal{S}_{\delta,\alpha^{k+1}}$.  Again using  \eqref{eq;IntRsRegBall} and Bishop-Gromov's volume monotonicity, we infer that
    \begin{align} \label{a}
    \nonumber
      \sum_{x_j \in \mathcal{S}_{\delta,\alpha^{k}} \setminus \mathcal{S}_{\delta,\alpha^{k+1}}}
    \int_{B_{r_j/10}(x_j)} |\R|^s \,\de\ssf{Vol}  \leq  
     10^{2s}\sum_{x_j \in \mathcal{S}_{\delta,\alpha^{k}} \setminus \mathcal{S}_{\delta,\alpha^{k+1}}}
   \ssf{Vol}(B_{r_j/10}(x_j)) r_j^{-2s} \\
    \leq 
   \sum_{x_j \in \mathcal{S}_{\delta,\alpha^{k}} \setminus \mathcal{S}_{\delta,\alpha^{k+1}}}
   c(n) \ssf{Vol}(B_{r_j/50}(x_j)) \alpha^{-2s(k+1)}.  
    \end{align}
    We note that the balls $\{B_{r_j/50}(x_j)\}_{j \in \bb{N}}$ are disjoint and 
    \[
    \bigcup_{x_j \in \mathcal{S}_{\delta,\alpha^{k}} \setminus \mathcal{S}_{\delta,\alpha^{k+1}}} B_{r_j}(x_j) \subset T_{\alpha^k}(\mathcal{S}_{\delta,\alpha^{k}} ).
    \]
    Hence, choosing $\eta:=1-s$ in Theorem \ref{CheegerNaberMain}, and combining with Proposition \ref{P2} and \eqref{a}, we deduce 
    \[
     \sum_{x_j \in \mathcal{S}_{\delta,\alpha^{k}} \setminus \mathcal{S}_{\delta,\alpha^{k+1}}}
    \int_{B_{r_j/10}(x_j)} |\R|^s \,\de\ssf{Vol}  \leq c(n,v,s) \alpha^{(1-s)k}.
    \]
    Hence,
    \[
    \sum_{k \in \bb{N}} \, \sum_{x_j \in \mathcal{S}_{\delta,\alpha^{k}} \setminus \mathcal{S}_{\delta,\alpha^{k+1}}}
    \int_{B_{r_j/10}(x_j)} |\R|^s \,\de\ssf{Vol} \leq c(n,v,s),
    \]
    giving the statement.
    \end{proof}
\end{thm}

One can now deduce   Jiang-Naber's result \cite[Theorem 2.17]{Naberconjectures}.

\begin{thm}[Jiang-Naber] \label{JiangNaber}
    Let $K \in \bb{R}$, $n \in \bb{N}$, $v>0$, $s \in (0,1)$ be fixed and let $(M^n,g,x)$ be a manifold with $\Ric_M \geq K$ and $\ssf{Vol}(B_1(x)) \geq v$. There exists a constant $C(K,n,v,s)>0$ such that
    \begin{equation}\label{eq:IntBDRics}
    \fint_{B_1(x)} |\Ric|^s \,\de\ssf{Vol} \leq  C(K,n,v,s).
    \end{equation}
    \begin{proof}
    By Theorem \ref{thm:IntRsB1Riccidelta} and a standard rescaling, there exists $r_0(K,n,v,s)>0$ such that  
    \begin{equation}\label{eq:IntRsBr0}
    \fint_{B_{r_0}(x)} |\R|^s \,\de\ssf{Vol} \leq  c(K,n,v,s), \quad \text{for all } y \in B_1(x).
    \end{equation}
    Consider a covering $\{B_{r_0}(x_j)\}_{j \in \bb{N}}$ of $B_1(x)$ such that the balls $\{B_{r_0/3}(x_j)\}_{j \in \bb{N}}$ are disjoint.  
    Then, the combination of \eqref{eq:IntRsBr0} with Bishop-Gromov's volume monotonicity yields
    \begin{align}\label{eq:fintRsBounded}
    \int_{B_1(x)} |\R|^s \, \de \ssf{Vol}& \leq \sum_{j \in \bb{N}} \int_{B_{r_0}(x_j)} |\R|^s \, \de \ssf{Vol}\leq c(n,K,v,s) \sum_{j \in \bb{N}} \ssf{Vol}(B_{r_0/3}(x_j)) \nonumber \\
    & \leq c(n,K,v,s) \ssf{Vol}(B_1(x)).
    \end{align}
    The assumption that $\Ric\geq K$ allows to promote the integral bound \eqref{eq:fintRsBounded} on the scalar cuvature into the claimed  integral bound \eqref{eq:IntBDRics} on the Ricci curvature.
    \end{proof}
\end{thm}

\textbf{Conflict of interest statement.} The authors have no relevant financial or non-financial interests to disclose. The authors
have no conflict of interest to declare that are relevant to the content of this article.
\smallskip

\textbf{Data Availability.} Data sharing is not applicable to this article as no new data were created or analyzed in this study.

\renewcommand*{\bibfont}{\normalfont\small}

\sloppy
%\tiny{
\printbibliography
%}

\end{document}